\documentclass[11pt]{article}
\usepackage{graphicx}
\usepackage{amssymb}
\usepackage{amsmath}
\usepackage{amsthm}
\usepackage{enumerate}
\usepackage{multicol}
\usepackage{verbatim}
\usepackage{verbatim}
\usepackage{subcaption}

\newtheorem{defn}{Definition}[section]
\newtheorem{definition}[defn]{Definition}
\newtheorem{thm}[defn]{Theorem}

\newtheorem{pr}[defn]{Proposition}

\newtheorem{lem}[defn]{Lemma}

\newtheorem{corollary}[defn]{Corollary}

\theoremstyle{definition}
\newtheorem{example}[defn]{Example}
\newtheorem{remark}[defn]{Remark}
\theoremstyle{remark}

\theoremstyle{example}

\def\grad{\nabla}
\def\vv{{\bf v}}

\def\zz{{\bf z}}
\def\xx{{\bf x}}
\def\one{{\bf 1}}
\def\rr{{\bf r}}

\def\codim{\operatorname{codim}}
\def\mm{{\bf m}}

\def\rhat{\hat{\rr}}

\def\sing{{V}}

\def\seq{{\mathbf{z}_n}}
\def\sequ{{\left\{\seq\right\}_{n=1}^{\infty}}}
\def\seqw{{\mathbf{w}_n}}
\def\sequw{{\left\{\seqw\right\}_{n=1}^{\infty}}}

\def\R{{\mathbb{R}}}

\def\CP{{\mathbb{CP}}}
\def\C{{\mathbb{C}}}
\def\NP{{\mathcal P}}
\def\FF{{\mathcal F}}
\def\AA{{\mathcal A}}

\begin{document}

\begin{center}
{\Large \bf Critical Points at Infinity for Hyperplanes of Directions} \\

{\large Stephen Gillen}

\end{center}

\textbf{Abstract}: Analytic combinatorics in several variables (ACSV) analyzes the asymptotic growth of the coefficients of a meromorphic generating function $F = G/H$ in a \textit{direction} $\mathbf{r}$.  It uses Morse theory on the pole variety $V := \{ H = 0 \} \subseteq (\C^*)^d$ of $F$ to deform the torus $T$ in the multivariate Cauchy Integral Formula via the downward gradient flow for the \textit{height} function $h = h_{\rr} = -\sum_{j=1}^d r_j \log |z_j|$, giving a homology decomposition of $T$ into cycles around \textit{critical points} of $h$ on $V$. The deformation can flow to infinity at finite height when the height function is not a proper map. This happens only in the presence of a critical point at infinity (CPAI): a sequence of points on $V$ approaching a point at infinity, and such that log-normals to $V$ converge projectively to $\mathbf{r}$.  The CPAI is called \textit{heighted} if the height function also converges to a finite value. This paper studies whether all CPAI are heighted, and in which directions CPAI can occur. We study these questions by examining sequences converging to faces of a toric compactification defined by a multiple of the Newton polytope $\NP$ of the polynomial $H$. Under generically satisfied conditions, any projective limit of log-normals of a sequence converging to a face $F$ must be parallel to $F$; this implies that CPAI must always be heighted and can only occur in directions parallel to some face of $\NP$. When this generic condition fails, we show under a smoothness condition, that a point in a codimension-1 face $F$ can still only be a CPAI for directions parallel to $F$, and that the directions for a codimension-2 face can be a larger set, which can be computed explicitly and still has positive codimension.

\section{Introduction}

A sequence {$a_n$} having combinatorial interest can often be encoded as a generating function, that is, as coefficients of a power series $\sum_{n=0}^\infty a_n z^n$ of a concisely described function $f(z)$. For example, the Fibonacci numbers have a generating function of $$\frac{1}{1-z-z^2} = 1 + 1z + 2z^2 + 3z^3 + 5z^4 + ....$$ It is often possible to use complex contour integration (based on the Cauchy Integral Formula) to get estimates on $a_n$ out of the function $f(z)$ represented by the power series; for rational functions, this can be done by examining the distance of its poles from the origin and their multiplicities. (For the Fibonacci numbers, the closest pole is $\frac{1}{\phi}$ of multiplicity 1, where $\phi = \frac{1+\sqrt{5}}{2}$ is the golden ratio, and the Fibonacci numbers grow asymptotically like a constant times $\phi^n$.) For univariate generating functions, there is a direct connection between the leading asymptotics of the coefficients of a univariate generating function and the locations and natures of its minimum-modulus singularities.

\textbf{Analytic combinatorics in several variables} (ACSV) is the science of doing this for \textit{multivariate} generating functions, that is, studying the asymptotic growth of the series coefficients of multivariate generating functions $f(z_1, ..., z_d) = \sum_\textbf{r} a_\textbf{r} \textbf{z}^\textbf{r}$ (where $\textbf{z}^\textbf{r}$ means $z_1^{r_1} \cdots z_d^{r_d}$) as $||\mathbf{r}|| \rightarrow \infty$ with $\frac{\mathbf{r}}{||\mathbf{r}||} \rightarrow \mathbf{\hat{r}}$ for some fixed unit vector $\mathbf{\hat{r}}$. For example, the binomial coefficients $\binom{n+k}{n}$ have the multivariate generating function $\frac{1}{1-x-y} = \sum_{k=0}^{\infty} \sum_{n=0}^{\infty} \binom{n+k}{n} x^n y^k$. Sometimes, the coefficient asymptotics of a more complicated univariate generating function (one given as the solution to an algebraic or differential equation) can be computed using the diagonal (the coefficients of $z_1^n \cdots z_d^n$) of a simpler (rational) generating function in several variables, so ACSV can help to analyze univariate generating functions as well.

ACSV relies heavily on methods from topology, complex analysis, and computer algebra. Given a rational generating function $F = G/H$, let $V \subseteq \C^d := \{ \zz : H(\zz) = 0 \}$ denote the affine variety on which $H$ vanishes, let $\C^* = \C \backslash \{0\}$, and let $V^* := V \cap (\C^*)^d$.  According to the multivariate Cauchy formula, if $T$ is a torus $$T = \{ \xx e^{i \theta} : \theta \in (\R / (2 \pi \R)^d \}$$ inside the domain of convergence for a Laurent series expansion of $H$, where $B$ is a convex subset of $\mathbb{R}^d$ that is disjoint from the image of $V^*$ under taking coordinate-wise log-modulus, then the coefficients $a_\rr$ are given by
$$a_\rr = \left ( \frac{1}{2 \pi i} \right )^d \int_T \zz^{-\rr - \one} F(\zz) \, d\zz \, .$$
Here, $\one$ represents the vector of all ones, $d\zz$ is the holomorphic volume form $dz_1 \wedge \cdots \wedge dz_d$, and we employ the multi-index notation $\mathbf{z}^{\mathbf{r}} = \prod_{j=1}^d z_j^{r_j}$.

The analysis involves deforming the torus $T$ as far from the origin as possible without crossing $V^*$, which is assumed to be smooth in this paper. This causes $T$ to decompose into cycles around certain \textit{critical points} of the height function $h_{\mathbf{r}}(\mathbf{z}) = -\sum_{j=1}^d r_j \log|z_j|$ on $V$; this is shown using Morse theory. (Morally, points farther from the origin have lower height because we are seeking the tightest possible bound on the asymptotics.) The problem ensues from the fact that Morse theory assumes that height functions are {\em proper}, meaning that the set of points whose height is in any compact interval must be compact. The next section details how Morse theory fails for nonproper functions.

\subsection{Critical points at infinity}

The gradient flow for $h_{\mathbf{r}}$ tells us how to deform $T$ downwards in height.  Morse theory assumes that $h : V^* \to \R$ is a proper function, meaning that the preimage of any compact set is compact.  When it is not, there is no guarantee of a decomposition of $T$ into cycles near critical points of $h$, and in fact we know examples where there is not such a decomposition.  Instead, under the downward gradient flow of $h$ on $V$, the cycle $T$ flows out to infinity before ever reaching a low enough height to encounter a critical point. (Note: For us, the coordinate hyperplanes, where one or more variables are zero, are also considered to be ``at infinity.") When this happens, there must be a critical point at infinity (CPAI).  To state this more precisely, we need a few definitions.  

\begin{defn}[Whitney stratification] ~~\\
	A Whitney stratified space is a finite collection of disjoint manifolds, called {\em strata}, in Euclidean space satisfying a containment condition and a Whitney condition.
	\begin{itemize}
		\item Containment: there is a partial order on the strata such that $S_{\alpha} < S_{\beta}$ if and only if $S_{\alpha} \subseteq \overline{S_{\beta}}$, where $\overline{S_{\beta}}$ denotes the topological closure of $S_{\beta}$.
		\item Second Whitney condition: if $S_{\alpha} < S_{\beta}$, $x_n \in S_{\beta}$, $y_n \in S_{\alpha}$, $x_n \to y \in S_{\alpha}$ and $y_n \to y$, then any limit of the secants $\overline{x_n y_n}$ is contained in the limit $\tau$ of the tangent spaces to $S_{\beta}$ at the $x_n$.  (This implies the first Whitney condition, that the tangent space to $S_{\alpha}$ at $y$ is also contained in $\tau$.)  For details, see any of the references~\cite{GM},~\cite[Appendix~D]{PW-book2} or~\cite{Whit1965}.
	\end{itemize}
\end{defn}

Any algebraic variety admits a Whitney stratification, and any complex algebraic variety admits a Whitney stratification whose strata are complex manifolds.

\begin{defn}[lognormal space]
	Let $H$ be a polynomial in $d$ variables, whose affine variety in the complex torus is denoted $V^* \subseteq (\C^*)^d$.  Let $\{ S_\alpha : \alpha \in I \}$ be a stratification of $V^*$ into complex manifolds, and let $\codim(S)$ denote the complex codimension of the stratum $S$.  For $\zz$ in any stratum $S$, let $T_\zz (S)$ denote the tangent space to $S$ at $\zz$ which has real codimension $2d - 2 \codim (S)$.  Define the lognormal space $N_\zz (V)$ to be the orthogonal complement in $\mathbb{R}^{2d}$ to $T_\zz (V)$ after mapping backward by the pointwise exponential map, then mapping forward again.   More formally, if $\phi : U \to W$ is the exponential map $\phi(\mathbf{z}) = \exp(\mathbf{z})$ on an open set in $(\C^*)^d$ such that $\phi$ is one to one, onto a neighborhood of $\zz$, then $N_\zz (S)$ is the image under $d\phi$ of the orthogonal complement in $\R^{2d} \cong \C^d$ of the tangent space to $\phi^{-1} (S)$ at $\phi^{-1} (\zz)$.
\end{defn}


\begin{defn}[CPAI]
	Fix a polynomial $H$ in $d$ variables and a Whitney stratification $\{ S_\alpha : \alpha \in I \}$ of its zero set $V^*$ in $(\C^*)^d$.
	Let $R \subseteq (\C^*)^d \times \CP^{d-1}$ be the relation holding for the pair $(\zz , \rr)$ if and only if $\zz \in S$ for some stratum $S$ and $\rr \in N_\zz (S)$.  A \textbf{CPAI} is a limit point $(\zz_* , \rr_*)$ of a sequence $(\zz_n , \rr_n)$ of pairs in the relation $R$ such that $\zz_* \notin \sing^*$.  In other words, $\zz_*$ must be a point at infinity, where infinity includes the coordinate hyperplanes as well as the hyperplanes at infinity.  We say that the sequence $(\zz_n , \rr_n)$ \textbf{witnesses} the CPAI.  A CPAI is called \textbf{heighted} if the sequence $h_{\rr_*} (\zz_n)$ has a finite limit point, and the set of such limit points for a given $\rr_*$ are called critical values at infinity (CVAI) in direction $\rr_*$ and denoted by $\beta (\rr_*)$.
\end{defn}

This definition, while a bit clunky, is precisely what is needed in~\cite{BaMP2022} to establish the Morse method for non-proper functions.  In particular, they show that in the absence of heighted CPAI with heights in an interval $[a,b]$, the usual Morse theoretic results hold over that interval.  For example, under this condition, if also there are no critical points for $h$ on $\sing^*$ with critical values in $[a,b]$, then there is a deformation retraction of the set $\sing^* \cap \{ \zz : h_{\rhat} (\zz) \leq b \}$ onto the set $\sing^* \cap \{ \zz : h_{\rhat} (\zz) \leq a \}$, hence these two spaces are homotopy equivalent.

A few important clarifications are necessary:

\begin{enumerate}
	\item All affine critical points of $h$ on the strata are defined by polynomial equations. Finding CPAI entails more than simply projectivizing the equations.  One must have an affine witness sequence, hence the algebra involves saturations of ideals.
	\item In \cite{BaMP2022}, the possibility is left open that a CPAI could exist but not be heighted, meaning that for any witness sequence, the heights of the points do not approach a finite number.
	\item Finally, it is possible for two sequences of points approaching criticality to converge to the same point in $\mathbb{CP}^d$ but have their heights converging to different values.
\end{enumerate}

The first of these three concerns is already addressed in \cite{BaMP2022}. This current work, which began at a Mathematics Research Communities (MRC) workshop, attempts to address the latter two issues by compactifying $V$ in the toric variety $X$ of the Newton polytope of $H$ instead of ordinary complex projective space: If $Q$ is the Newton polytope of $H$, enlarged to make it normal (see Definition \ref{normaldefn}), and $\left\{\mathbf{m}_1, ..., \mathbf{m}_s\right\} := Q \cap \mathbb{Z}^d$, then $X$ is defined to be the Zariski closure of the image of the map $\Phi : (\mathbb{C} \backslash \{0\})^d \rightarrow \mathbb{CP}^{s-1}$ given by $\Phi(\mathbf{t}) = \left[\mathbf{t}^{\mathbf{m}_1} : \cdots : \mathbf{t}^{\mathbf{m}_s}\right]$. 

Throughout the rest of this paper we assume that $V$ is smooth; $V$ itself forms a Whitney stratification. Under certain generically satisfied conditions, compactifying in $X$ guarantees that the only directions for which CPAI can exist are those parallel to a face of the Newton polytope, and that if the images in $X$ of two CPAI witness sequences converge to the same point in this new compactification, then the heights for the corresponding height function converge to the same finite number. These results also lead naturally to the conclusion that CPAI (as defined in \cite{BaMP2022}) correspond to points at infinity with well-defined heights in the toric compactification.  In addition, when the geometry of the Newton polytope satisfies certain conditions, this compactification admits a simplified coordinate system near any point at infinity, which is much easier to work with than limits of sequences of points in affine space.

\section{Toric Variety Background and Height Convergence}

\subsection{Background: The Newton Polytope and \\ Normality}

\begin{defn}
	A \textbf{convex polytope} can be defined either as the convex hull of a finite set of points in $\mathbb{R}^d$, or as a bounded intersection of finitely many half-spaces.
\end{defn}

\begin{defn}
	Let $H$ be a Laurent polynomial in $d$ variables, a finite sum of terms of the form $\sum_{j=1}^n c_j \zz^{\mathbf{m}_j}$ where $c_j \neq 0$ and $\mm_j \in \mathbb{Z}^d$. Then the \textbf{Newton polytope} of $H$ is the convex hull of the exponent vectors $\mm_j$.
\end{defn}

\begin{defn}
	\label{normaldefn}
	A convex polytope $Q \subseteq \mathbb{R}^d$ is called a \textbf{lattice polytope} if all of its vertices are in $\mathbb{Z}^d$. (For example, the Newton polytope of a Laurent polynomial $H$ is a lattice polytope.) A lattice polytope $Q$ is called \textbf{normal} if, for every positive integer $k$, every point in $kQ \cap \mathbb{Z}^d$ can be written as a sum of exactly $k$ points in $Q \cap \mathbb{Z}^d$. (If this is true for all sufficiently large integers $k$, but not necessarily for every positive integer $k$, then $Q$ is called \textbf{very ample}.)
\end{defn}

In the work that follows, we adopt the following setup: Let $H$ be a Laurent polynomial in $d$ variables with complex coefficients. Let $\NP$ be its Newton polytope. Assume WLOG that $\NP$ has full dimension $d$; otherwise, a monomial change of coordinates reduces to this case.  

Let $Q = \kappa\NP$ be an enlargement of $\NP$ that is normal, where $\kappa$ is a positive integer.  By~\cite[Theorem~2.2.11]{CLS}, for any $d > 1$, the scaling factor $\kappa = d-1$ is always sufficient.  Faces of normal polytopes are normal, and normality of $Q$ implies $Q$ is very ample. Let $\AA = \NP \cap \mathbb{Z}^d$, and let $A = Q \cap \mathbb{Z}^d = \{\mm_1, ..., \mm_s\}$ be the set of integer lattice points on and inside $Q$. We define a map $\Phi : (\mathbb{C}^*)^d \rightarrow \mathbb{CP}^{s-1}$ given by $\Phi(\mathbf{z}) = [\zz^{\mm_1} : \cdots : \zz^{\mm_s}]$ and define $X_A$ to be the closure of the image of $\Phi$ in $\mathbb{CP}^{s-1}$.  The topological and Zariski closures are equal by GAGA (Proposition 7 in \cite{Serr1956}).  For ACSV purposes, we will primarily be using the fact that it is the topological closure. The set $X_A$ is a closed subset of $\mathbb{CP}^{s-1}$ and is therefore compact, so every sequence has a convergent subsequence. Let $\sequ$ be a sequence of points in $(\mathbb{C}^*)^d$ such that their images $\Phi(\seq) \in X_A$ converge to a point $p$. We are chiefly concerned with the case where $\zz_n$ does not converge to an affine point, that is, the sequence ``goes to infinity."

\subsection{Classifying points in $X_A$ into faces at infinity}

We now describe what it means for a point $p \in X_A \subseteq \mathbb{CP}^{s-1}$ to be in the ``face at infinity" corresponding to a face $F$ of the enlarged Newton polytope $Q$. This is called the ``torus orbit" corresponding to $F$ in the language of \cite{GKZ}, but we believe the term ``face at infinity" is a lot more intuitive for ACSV purposes, at least when $F$ is not all of $Q$.

\begin{defn}
	Suppose that $F$ is a face of $Q$. Let $X^0(F)$, the interior of the face at infinity corresponding to $F$, be the set of all points $p$ in $X_A$ such that the nonzero coordinates of $p$ in $\mathbb{CP}^{s-1}$ are precisely those coordinates $j$ corresponding to lattice points $\mathbf{m}_j$ in $\mathbb{Z}^d \cap F$. Let $X(F)$ denote the closure of $X^0 (F)$ in $\CP^{s-1}$, or equivalently, in $X_A$.  We call $X(F)$ the face at infinity corresponding to $F$.
\end{defn}

Notice that a point $p \in X(F)$ must have coordinates equal to zero in all positions corresponding to lattice points outside $F$, but some of the coordinates of $p$ in positions corresponding to lattice points in $F$ may also be zero. Also, when $F$ is all of $Q$ (which is considered a face of itself), $X(Q)$ is all of $X_A$, and $X^0(Q)$ is the set of points $p$ in $X_A$ whose coordinates in $\mathbb{CP}^{s-1}$ are all nonzero.

An important result from the literature is that every point in $X_A$ is in $X^0(F)$ for exactly one face $F$ of $Q$.

\begin{lem}[Faces at infinity]
	\label{faces}
	Let $p \in X_A$. Then $p$ is in the interior of the face at infinity corresponding to exactly one face $F$; that is, there exists a unique face $F$ of $Q$ (of some dimension) such that the nonzero coordinates of $p$ are precisely those in positions corresponding to lattice points in $\mathbb{Z}^d \cap F$. (If $p$ has all nonzero coordinates, then $F = Q$.) Furthermore, for every face $F$, $X(F) \cong X_{A \cap F}$ is nonempty and has dimension equal to the dimension of $F$.
\end{lem}
\begin{proof}
	See Proposition 1.9 of Chapter 5 of \cite{GKZ} and Proposition 2.1.6(b) of \cite{CLS}. 
\end{proof}

\subsection{Which monomials converge as we move toward a face at infinity?}

Let $F$ be the face of $Q$ such that $p$ resides in the interior of the face at infinity corresponding to $F$, and let $\FF$ be the corresponding face of the original Newton polytope $\NP$. In addition, it will be useful to have notation for the polytope $Q$ and the face $F$ with one of the vertices of $F$ shifted to the origin, as well as for the convex cone of all directions that point from $F$ into $Q$.

\begin{definition}
	Let $\mathbf{v}$ be a vertex of $F$, and let $\mathfrak{v}$ be the corresponding vertex of $\NP$ (so that $\vv = \kappa \mathfrak{v}$). We define the following notation:
	\begin{enumerate}
		\item $\tilde{Q} = Q - \mathbf{v}$,
		\item $\tilde{F} = F - \mathbf{v}$,
		\item $\tilde{\NP} = \NP - \mathfrak{v}$,
		\item $\tilde{\FF} = \FF - \mathfrak{v}$,
		\item $L_F$ is the linear span of $\tilde{F}$ (this is the span of differences of points in $F$ and does not depend on which vertex $\mathbf{v}$ was chosen),
		\item $\pi(\tilde{Q})$ is the image of $\tilde{Q}$ in the quotient space $\mathbb{R}^d/L_F$ (which does not depend on $\mathbf{v}$ because the difference between any two of them is in $L_F$), and
		\item $\sigma_F$ is the nonnegative linear span of the preimage (under the quotient map) of $\pi(\tilde{Q})$ (and also does not depend on $\vv$).
	\end{enumerate}
\end{definition}

Regardless of the choice of $\mathbf{v}$, $\sigma_F$ is the set of all vectors that can be written as $\mathbf{x} + \mathbf{y}$ for some $\mathbf{x} \in L_F$ and $\mathbf{y} \in k\tilde{Q}$ for some sufficiently large positive integer $k$. Note that If $F$ has codimension 1, then $\sigma_F$ is a half-space.


\begin{lem}
	\label{convex}
	Suppose that the $k$ points $\mm_1, ..., \mm_k \in Q$ satisfy $\mm_1 + \cdots + \mm_k \in kF$. Then $\mm_1, ..., \mm_k \in F$.
\end{lem}
\begin{proof}
	If $F = Q$ this is obvious, so we will assume $F \subsetneq Q$. Because $F$ is a proper face of $Q$, it has a supporting hyperplane whose intersection with $Q$ is exactly $F$, and such that for some $c \in \mathbb{R}$, the outward-pointing normal vector $N$ to this hyperplane satisfies $N \cdot \mm = c$ for all $\mm \in F$, and $N \cdot \mm < c$ for all $\mm \in Q \backslash F$. We have that $\frac{\mm_1 + \cdots + \mm_k}{k} \in F$, so $\frac{N \cdot \mm_1 + \cdots + N \cdot \mm_k}{k} = c$. But every term of the sum in the numerator is at most $c$, so all must be exactly $c$.
\end{proof}

\begin{lem}
	\label{normality}
	Suppose $F$ is a face of the normal convex lattice polytope $Q$. ($F$ could be all of $Q$, or a face of positive codimension.) Then any integer vector $\mathbf{x} \in L_F \cap \mathbb{Z}^d$ can be written as an integer linear combination of integer vectors in $\tilde{F} \cap \mathbb{Z}^d$, or alternatively as an integer linear combination of lattice points in $F \cap \mathbb{Z}^d$ with coefficients adding up to zero.
\end{lem}
\begin{proof}
	Let $g$ be the dimension of $F$. Then it is possible to choose vertices $\mathbf{v}_0, \mathbf{v}_1, ..., \mathbf{v}_g$ of $F$ such that $\{\mathbf{v}_k - \mathbf{v}_0 : 1 \leq k \leq g\}$ is a basis for $L_F$. Therefore, $\mathbf{x}$ can be recovered as a rational linear combination of the vectors $\mathbf{v}_k - \mathbf{v}_0 \in \tilde{F} \cap \mathbb{Z}^d$ for $1 \leq k \leq g$, or alternatively, a rational linear combination of the $\mathbf{v}_k \in F \cap \mathbb{Z}^d$ for $0 \leq k \leq g$ with the coefficients adding up to zero. Clearing denominators and rearranging terms to get plus signs, we get that there exist nonnegative integers $a_k, b_k$ (for $0 \leq k \leq g$) and a positive integer $c$ such that 
	
	$$\left(\sum_{k=0}^g a_k \mathbf{v}_k\right) + c \mathbf{x} = \left(\sum_{k=0}^g b_k \mathbf{v}_k\right)$$
	
	and such that $(\sum_{k=0}^g a_k) = (\sum_{k=0}^g b_k)$. Let $K$ be this common sum.
	
	By convexity of $KF$, we have the following:
	$$\left(\sum_{k=0}^g a_k \mathbf{v}_k\right) \in KF$$
	$$\left(\sum_{k=0}^g a_k \mathbf{v}_k\right) + c \mathbf{x} = \left(\sum_{k=0}^g b_k \mathbf{v}_k\right) \in KF$$
	Therefore, again by convexity,
	$$\left(\sum_{k=0}^g a_k \mathbf{v}_k\right) + \mathbf{x} \in KF.$$
	
	By normality of the polytope $Q$, $\left(\sum_{k=0}^g a_k \mathbf{v}_k\right) + \mathbf{x}$ can be written as the sum of exactly $K$ lattice points $\mathbf{w}_1, ..., \mathbf{w}_K$ in $\mathbb{Z}^d \cap Q$, and by Lemma \ref{convex}, these $K$ lattice points are actually in $\mathbb{Z}^d \cap F$. (These are not necessarily vertices of $F$.) Therefore, subtracting $\sum_{k=0}^g a_k \mathbf{v}_k$, we have that
	
	$$\mathbf{x} = \sum_{j=1}^K \mathbf{w}_j - \sum_{k=0}^g a_k \mathbf{v}_k$$
	
	can be written as an \textit{integer} linear combination of lattice points in $\mathbb{Z}^d \cap F$ with coefficients adding up to $K - \sum_{k=0}^g a_k = 0$, so if $\mathbf{v}$ is any vertex of $F$, then
	
	$$\mathbf{x} = \sum_{j=1}^K (\mathbf{w}_j - \mathbf{v}) - \sum_{k=0}^g a_k (\mathbf{v}_k - \mathbf{v})$$
	
	is an integer linear combination of integer vectors in $\tilde{F} \cap \mathbb{Z}^d$ as desired.
	
\end{proof}

\begin{corollary}
	When $Q$ is normal, the map $\Phi : (\mathbb{C}^*)^d \rightarrow \mathbb{CP}^{s-1}$ is injective on $(\mathbb{C}^*)^d$. [Note: $\Phi$ is not necessarily injective on all of $\mathbb{C}^d$, if it is even defined there.]
\end{corollary}
\begin{proof}
	Setting $F = Q$ in Lemma \ref{normality} and noting that $L_Q$ is all of $\mathbb{R}^d$, we see that the standard basis vector $\mathbf{e}_1$ can be written as an integer linear combination of lattice points in $Q$ with coefficients adding up to zero, so $\mathbf{e}_1 = \sum_{j=1}^s n_j \mm_j$, with $\sum_{j=1}^s n_j = 0.$
	
	Now suppose $\mathbf{p} \in \Phi((\mathbb{C}^*)^d)$. Then it is of the form $[C \mathbf{t}^{\mathbf{m}_1} : ... : C \mathbf{t}^{\mathbf{m}_s}]$, where $C, t_1, ..., t_d \neq 0$. Setting $\mathbf{n} = (n_1, ..., n_s)$, we can see that $$\mathbf{p}^\mathbf{n} = C^{\sum_{j = 1}^s n_j} \mathbf{t}^{\sum_{j = 1}^s n_j \mathbf{m}_j} = C^0 \mathbf{t}^{\mathbf{e}_1} = t_1,$$ regardless of $C$. We can similarly recover the remaining coordinates $t_2, ..., t_d$ (repeating the process with $\mathbf{e}_1$ replaced with each of the remaining standard basis vectors), and the point in $(\mathbb{C}^*)^d$ that maps to $\mathbf{p}$ is therefore unique.
\end{proof}

It will be helpful to view $(\mathbb{C}^*)^d$ as a subset of $X_A$, with the injection $\Phi$ acting as an embedding. All Laurent monomials are defined at points in $(\mathbb{C}^*)^d$; it will be helpful to determine which Laurent monomials can be extended continuously to the interior of certain faces at infinity in $X_A$. The following lemma defines the value of certain monomials at points at infinity $p \in X_A$ by taking limits of monomials evaluated at points on sequences $\sequ$ in $(\mathbb{C}^*)^d$ such that $\Phi(\seq)$ converges to $p$. Specifically, if $p$ is in the interior of the face at infinity in $X_A$ corresponding to a face $F$ of $Q$, then $p^{\mathbf{m}}$ can be defined whenever $\mathbf{m} \in \sigma_F$, and it is nonzero whenever $\mathbf{m} \in L_F$. Note that the notation $p^{\mathbf{m}}$ does not conflict with the usual multinomial power notation because $p \in \mathbb{CP}^{s-1}$ and $\mathbf{m} \in \mathbb{Z}^d$ have $s$ and $d$ coordinates respectively, and $s > d$ whenever $Q$ has dimension $d$.

\begin{thm}[Which Laurent monomials converge]
	\label{monomials}
	Suppose that $\sequ$ is a sequence in $(\mathbb{C}^*)^d$ with $\Phi(\seq)$ converging to a point $p \in X_A$ in the interior of the face at infinity corresponding to $F$.
	\begin{enumerate}
		\item Then $\mathbf{z}_n^{\mathbf{m}}$ converges to a finite nonzero value $p^{\mathbf{m}}$ (dependent on $p$ but not on the sequence) for all integer vectors $\mm \in \tilde{F}$, and if $\mm$ is any other integer vector in $\tilde{Q}$, then $\mathbf{z}_n^{\mathbf{m}}$ converges to $p^{\mathbf{m}} = 0$.
		\item In fact, $\mathbf{z}_n^{\mathbf{m}}$ converges to a finite nonzero value $p^{\mathbf{m}}$ (dependent on $p$ but not on the sequence) for all integer vectors $\mm \in L_F$, and if $\mm$ is any other integer vector in $\sigma_F$, then $\mathbf{z}_n^{\mathbf{m}}$ converges to $p^{\mathbf{m}} = 0$.
		
	\end{enumerate}
\end{thm}
\begin{proof}
	Let $\mathbf{v}$ be one of the vertices of $F$, and assume without loss of generality that $\{\mm_1, ..., \mm_s\}$ is ordered so that $\mm_s = \mathbf{v}$. We first show part 1, that $\mathbf{z}_n^{\mathbf{m}}$ converges to a finite value for all $\mm$ in $\tilde{Q} = Q - \mathbf{v}$, and that value is nonzero if and only if $\mm$ is in $\tilde{F} = F - \mathbf{v}$.
	
	Let $P_1 = [\mathbf{z}_1^{\mathbf{m}_1} : ... : \mathbf{z}_1^{\mathbf{m}_s}], P_2 = [\mathbf{z}_2^{\mathbf{m}_1} : ... : \mathbf{z}_2^{\mathbf{m}_s}], ...$ be a sequence in the image of $\Phi$ that converges to $p \in \mathbb{CP}^{s-1}$. The point $p$ is in the face at infinity corresponding to $F$, so the last coordinate in $p$ (which corresponds to $\mathbf{v} = \mm_s$) is nonzero, so the same is true of $\Phi(\seq)$ for sufficiently large $n$; we can therefore choose the chart of $\mathbb{CP}^{s-1}$ in which we divide by the last coordinate. If we represent $p \in \mathbb{CP}^{d-1}$ as $p = [P_1 : ... : P_{s-1} : 1]$, we get that
	
	$$\left(\mathbf{z}_n^{\mathbf{m}_1 - \mathbf{v}}, ..., \mathbf{z}_n^{\mathbf{m}_s - \mathbf{v}}\right) \rightarrow (P_1, ..., P_{s-1}, 1).$$
	
	Because $p \in X^0(F)$, the nonzero coordinates of $p$ are precisely those in positions corresponding to lattice points in $\mathbb{Z}^d \cap F$, so $\mathbf{z}_n^{\mathbf{m}_j - \mathbf{v}}$ converges to a nonzero number when $\mathbf{m}_j \in F$ (equivalently, if $\mathbf{m}_j - \mathbf{v} \in \tilde{F}$), and zero otherwise. All of the lattice points in $\tilde{Q}$ are $\mm_j - \mathbf{v}$ for some $j \in \{1, ..., s\}$, so we see that for all $\mm \in \tilde{Q}$, $\mathbf{z}_n^{\mathbf{m}}$ converges to a finite value for all $\mm$ in $\tilde{Q}$, and that value is nonzero if and only if $\mm$ is in $\tilde{F}$. Furthermore, this value depends only on the coordinates of $p$ and not on the sequence used to approach $p$.
	
	For part 2, suppose that $\mm \in \mathbb{Z}^d \cap \sigma_F$. If $g$ is the dimension of $F$, then it is possible to choose vertices $\mathbf{v}_0, \mathbf{v}_1, ..., \mathbf{v}_g$ of $F$ such that $\{\mathbf{v}_k - \mathbf{v}_0 : 1 \leq k \leq g\}$ is a basis for $L_F$. Then $\mm$ can be written as $\mathbf{x} + \mathbf{y}$ for some $\mathbf{x} \in L_F$ and $\mathbf{y} \in k\tilde{Q}$ (where, in $\tilde{Q}$, we can choose $\mathbf{v} = \mathbf{v}_0$ without loss of generality) for some sufficiently large positive integer $k$. We need to show that $\mathbf{x}$ and $\mathbf{y}$ can be chosen to be \textit{integer} vectors. The vector $\mathbf{x} \in L_F$ can be written as a rational linear combination of vectors in $\tilde{F}$ as $\mathbf{x} = \sum_{j=1}^g a_j (\mathbf{v}_j - \mathbf{v})$. Now, simply let $\mathbf{x}_0 = \sum_{j=1}^g b_j (\mathbf{v}_j - \mathbf{v})$, where $b_j$ is the greatest integer less than or equal to $a_j$. Then $\mathbf{x}_0$ is an integer linear combination of vectors in $\tilde{F}$, so $\mathbf{x}_0 \in L_F \cap \mathbb{Z}^d$, and $\mathbf{x} - \mathbf{x}_0$ is a nonnegative linear combination of vectors in $\tilde{F} \subseteq \tilde{Q}$, so $\mathbf{x} - \mathbf{x}_0$ and also $\mathbf{y}_0 = \mathbf{x} - \mathbf{x}_0 + \mathbf{y}$ are in the nonnegative linear span of $\tilde{Q}$ and therefore in $k_0\tilde{Q}$ for some sufficiently large positive integer $k_0$. Because $\mathbf{x}_0 + \mathbf{y}_0 = \mathbf{x} + \mathbf{y} = \mm$, $\mathbf{y}_0$ is also an integer vector. Furthermore, because $\mathbf{x}_0 \in L_F$, we have that $\mm \in L_F$ if and only if $\mathbf{y}_0 \in L_F \cap k_0\tilde{Q} = k_0\tilde{F}$. (Any vector in $L_F$ will satisfy the equation of a supporting hyperplane of $\tilde{F}$ as a face of $\tilde{Q}$.)
	
	Next, we show that $\mathbf{z}^{\mathbf{x}_0}$ converges to a nonzero constant. By Lemma \ref{normality}, $\mathbf{x}_0$ can be written as $\sum_{j=1}^K c_j \mathbf{w}_j$ for some positive integer $K$, integers $c_j$, and integer vectors $\mathbf{w}_j \in \tilde{F} \cap \mathbb{Z}^d$. Therefore, the Laurent monomial $\mathbf{z}^{\mathbf{x}_0}$ can be written as a product $\prod_{j=1}^K (\mathbf{z}^{\mathbf{w}_j})^{c_j}$, where the monomials $\mathbf{z}^{\mathbf{w}_j}$ have been shown in part 1 to converge to finite nonzero values independent of the sequence chosen.
	
	Finally, because $\mathbf{y}_0$ is in $k_0\tilde{Q}$ for some positive integer $k_0$, then normality implies that it can be written as a sum of $k_0$ integer points $\mathbf{y}_j \in \tilde{Q}$. By Lemma \ref{convex}, they are all in $\tilde{F}$ if and only if $\mathbf{y}_0$ is in $k_0\tilde{F}$. Writing $\mathbf{y}_0 = \sum_{j=1}^{k_0} \mathbf{y}_j$, we see by part 1 that the product $\mathbf{z}^{\mathbf{y}_0} = \prod_{j=1}^{k_0} \mathbf{z}^{\mathbf{y}_j}$ will have all factors approaching finite values independent of the sequence, and the values will be all nonzero if and only if each $\mathbf{y}_j \in \tilde{F}$, which is true if and only if $\mathbf{y}_0$ is in $k_0\tilde{F}$ (that is, when $\mm \in L_F$).
	
	Therefore, $\mathbf{z^m} = \mathbf{z}^{\mathbf{x}_0} \mathbf{z}^{\mathbf{y}_0}$ approaches a nonzero finite value that depends on $p$ alone if $\mm \in L_F$, and zero when $\mm \in \sigma_F \backslash L_F$.
\end{proof}

WARNING: Not all monomials either converge for all sequences whose limit is $p$ or go to zero or infinity. For some, convergence behavior depends on the sequence. Later, we will see an example of a ``phantom" heighted CPAI in an unexpected direction, where the log-gradient direction does not converge to being parallel to the face.

An ACSV application immediately follows, namely the continuous extension of the height function $h_{\mathbf{r}}$ to $X^0(F)$ provided that $\mathbf{r} \in L_F$ (or in simpler language, provided that $\mathbf{r}$ is parallel to $F$).

\begin{corollary}
	\label{heightconf}
	Under the hypotheses of Lemma \ref{monomials}, whenever $\mathbf{r} \in \mathbb{Z}^d$ is parallel to $F$, the height function $h_{\mathbf{r}}(\seq)$ converges to a finite value that depends on $p$ but not on the sequence chosen.
\end{corollary}
\begin{proof}
	The height function $h_{\mathbf{r}}(\mathbf{z})$ is given by $-\sum_{j=1}^d r_k\log|z_k| = -\log|\mathbf{z^r}|$, so $h_{\mathbf{r}}(\seq)$ converges whenever the monomial $\mathbf{z}_n^{\mathbf{r}}$ converges to a finite nonzero value.
\end{proof}

Notice that we did not actually need normality for the convergence of the height function; we could get convergence of the modulus even if we only had the exponent vector as a \textit{rational} linear combination of the exponent vectors of monomials in $F$ with coefficients adding up to zero. However, we did need normality for the injectivity of $\Phi$; otherwise, the preimage of a point in the image of $\Phi$ could have finite cardinality greater than 1, which would mean that $\Phi$ could not be thought of as an embedding of $\sing^*$ into $X_A$, and $X_A$ would not be a compactification of $\sing^*$ in the traditional sense. Also, the convergence of the Laurent monomials themselves (not just their moduli) will come in handy when we examine convergence of the log-gradient vector, each of whose components is a linear combination of these monomials.

Finally, even if $\mathbf{r}$ has rational components and the monomial $\mathbf{z^r}$ is multivalued, its modulus is well-defined by taking positive real roots, so we also get convergence of the height function to a finite nonzero value when $\mathbf{r} \in \mathbb{Q}^d$ is parallel to $F$. Therefore, because any real number can be approximated by rationals, $h_{\mathbf{r}}$ converges to a finite nonzero value for all $\mathbf{r} \in \mathbb{R}^d$ that are parallel to $F$. (Note: The asymptotics in an irrational direction are defined as the limit of the asymptotics in neighboring directions. This idea will come up again later when considering the limitations of reducing the number of variables.)

\section{Log-Gradient Convergence: Generic Smooth Case}

Let $H$ be a polynomial, and let $V \subseteq \mathbb{C}^d$ be its zero set. The \textbf{log-gradient} of $H$, denoted $\nabla_{\log} H$, is defined as $\left(z_1 \frac{\partial H}{\partial z_1}, ..., z_d \frac{\partial H}{\partial z_d}\right)$ and is the normal to the hypersurface $H=0$ when the coordinates are changed to the natural logarithms of the original variables. If $V^* = V \cap (\mathbb{C}^*)^d$ is smooth, a critical point at infinity exists in the direction $\mathbf{r}$ when there is a sequence $\sequ$ with one or more coordinates approaching either zero or infinity, such that the projective directions of the log-gradient vectors approach the direction of $\mathbf{r}$. By compactness, there exists a subsequence such that $\Phi(\seq)$ converges in $X_A$ to some point $p$.

Multiplying $H$ by a monomial factor does not change the projective direction of the log-gradient vector for points in $(\mathbb{C}^*)^d$ satisfying $H = 0$. (However, it may change the log-gradient direction for points that do not satisfy $H = 0$.)

\begin{lem}
	\label{monomgradlog}
	Let $F$ be a multivariate polynomial, and let $m$ be a Laurent monomial. Then for points on the variety defined by $F$ with all nonzero coordinates, the log-gradient of $mF$ is projectively equivalent to the log-gradient of $F$; in fact, $\grad_{\log} (mF) = m \grad_{\log} F$ for points on $\sing^*$.
	
\end{lem}
\begin{proof}
	We begin with the case in which $m$ is a single variable, say $z$. In this case, all coordinates besides the $z$-coordinate of the log-gradient of $zF$ are simply $z$ times the corresponding coordinate of the log-gradient of $F$. The $z$-component of the log-gradient of $F$ is $z \frac{\partial}{\partial z} (zF) = z^2 \frac{\partial F}{\partial z} + zF \cdot 1$ by the product rule. For points on the variety, $F$ evaluates to zero, so the result is equal to $z$ times the log-gradient of $F$ for points on the variety. Therefore, the log-gradient of $zF$ is projectively equivalent to the log-gradient of $F$ (specifically, $z$ times the log-gradient of $F$) for points on $V(F) \cap (\mathbb{C}^*)^d$, and (replacing $F$ with $z^{-1} F$) the same is true for multiplication by $z^{-1}$.
	
	Iterating this procedure gives the result for all Laurent monomials.
\end{proof}

\subsection{Standard Rescaling: Motivation}

It will frequently be the case that the log-gradient vector will have magnitude approaching infinity as the sequence goes to infinity with images in $X_A$ converging to $p$. So in order to show something about the limit of the direction of the log-gradient, we will need to focus on a particular way of rescaling the log-gradient vector to look at a limit in projective space; we will divide through all components of the log-gradient by a monomial $\mathbf{z}^{\mathfrak{v}}$ corresponding to one of the vertices of the original Newton polytope $\NP$. To understand why the rescaling we choose is (usually) correct, we take a look at a very simple example.

\begin{example}
	Suppose $H = 1 + x + x^2 + xy + x^2 y$, so that its log-gradient is given by the vector 
	$$(x + 2x^2 + xy + 2x^2 y, xy + x^2 y).$$
	In this example, the Newton polytope is the quadrilateral with vertices (0,0), (1,1), (2,1), and (2,0) and is already normal, and the equation $H=0$ can be solved for $y$ to give $y = \frac{-1-x-x^2}{x(1+x)}$. Suppose that $x$ is approaching -1, so that $y$ approaches infinity, and $\Phi(\mathbf{z}) = [1 : x : x^2 : xy : x^2 y]$ approaches $p = [0 : 0 : 0 : 1 : -1] \in X_A$, located on the face at infinity corresponding to the top edge of $\NP$; we denote this edge by $\mathcal{F}$. In our example, we simply divide through by $xy$ to produce $[y^{-1} + 2xy^{-1} + 1 + 2x : 1 + x]$ as a projective direction, which approaches $[0 + 0 + 1 + 2(-1) : 1 + (-1)] = [-1 : 0]$. (This is indeed parallel to the horizontal edge $\mathcal{F}$.)
\end{example}

 The idea is that, as we grow closer to $p$, monomials on the edge $\mathcal{F}$ (in this case, $xy$ and $x^2 y$) are ``dominant" compared to monomials in the rest of $\NP$ (in this case, 1, $x$, and $x^2$). Therefore, dividing every term by one of these ``dominant" monomials (say, $\mathbf{z}^{\mathfrak{v}}$ for $\mathfrak{v}$ some vertex of $\mathcal{F}$) should cause all terms to approach either zero or a nonzero constant, and because the log-gradient has the same monomials in it as the original polynomial (except for the lack of a constant term), there will exist monomials that do not approach zero. The only problem that could arise is that, with just the right choices of coefficients for $H$, the contributions of those ``dominant" monomials could happen to cancel, leaving the zero vector, meaning that we rescaled too far and have lost the information about the limiting projective direction. When a ``singularity at infinity" like this occurs, we have examples that have a positive-dimensional set of possible limiting projective directions for the log-gradient, depending on the sequence chosen (more on this later). But generically, there will be at least one component where the contributions of the ``dominant" monomials do not exactly cancel, and in such cases, it is sufficient to show that this rescaling of the log-gradient can be written as the sum of two components: one that is always parallel to $\mathcal{F}$ for every $\seq$, and one that approaches the zero vector. 

\subsection{If $\zz^{-\mathfrak{v}} \grad_{\log} H$ approaches a nonzero vector, then the limiting direction is parallel to F.}

As before, suppose that $\Phi(\seq)$ converges to a point $p \in X_A$ in the interior of the face at infinity corresponding to some face $F$ of $Q$ (and the corresponding face $\mathcal{F}$ of $\NP$). We choose a vertex $\mathfrak{v}$ of $\mathcal{F}$, let $\mathbf{v} = \kappa \mathfrak{v}$ be the corresponding vertex of $F$, and define the rescaled log-gradient as $\mathbf{z}^{-\mathfrak{v}} \nabla_{\log} H$. By the earlier lemmas, each of the monomials in $\mathbf{z}^{-\mathfrak{v}} \nabla_{\log} H$ will approach either zero or a nonzero constant, and for generic $H$, $\mathbf{z}^{-\mathfrak{v}} \nabla_{\log} H$ will approach a finite nonzero vector in the limit. When this occurs, we can show that this limiting vector is parallel to $F$.

\begin{thm}
	\label{genericresult}
	Suppose $\sequ$ is a sequence in $\sing^*$ such that $\Phi(\seq)$ converges to a point $p$ in the interior of the face at infinity corresponding to a face $F$ of $Q$. ($F$ need not be a facet, and it is impossible for it to be a vertex.) If $\mathbf{z}^{-\mathfrak{v}} \nabla_{\log} H$ does not approach the zero vector, then its limit is parallel to $F$.
\end{thm}
\begin{proof}
	First, notice that the log-gradient of a monomial is simply $$\nabla_{\log} \mathbf{z^m} = (m_1 \mathbf{z^m}, ..., m_d \mathbf{z^m}) = \mathbf{z^m m}.$$ So if $H = \sum_{\mathbf{m} \in \NP} c_{\mathbf{m}} \mathbf{z^m}$, then $\mathbf{z}^{-\mathfrak{v}} \nabla_{\log} H = \mathbf{z}^{-\mathfrak{v}} \sum_{\mathbf{m} \in \NP} c_{\mathbf{m}} \mathbf{z^m m} = \sum_{\mathbf{m} \in \NP} c_{\mathbf{m}} \mathbf{z^{m-\mathfrak{v}} m}$.
	
	Let $H_{\mathcal{F}}$ be the component of $H$ consisting of all terms whose monomials are in $\mathcal{F}$, so that $H - H_{\mathcal{F}}$ has all of the terms whose monomials are in $\NP \backslash \mathcal{F}$. We split this up into two components as follows:
	
	$C_1 = \sum_{\mathbf{m} \in \mathcal{F}} c_{\mathbf{m}} \mathbf{z^{m-\mathfrak{v}} m} + \mathbf{z^{-\mathfrak{v}}} (H - H_{\mathcal{F}}) \mathfrak{v}$
	
	$C_2 = \sum_{\mathbf{m} \in \NP \backslash \mathcal{F}} c_{\mathbf{m}} \mathbf{z^{m-\mathfrak{v}} m} - \mathbf{z^{-\mathfrak{v}}} (H - H_{\mathcal{F}}) \mathfrak{v}$
	
	It is clear that $C_1 + C_2 = \sum_{\mathbf{m} \in \NP} c_{\mathbf{m}} \mathbf{z^{m-\mathfrak{v}} m}$. Two claims must now be justified:
	
	\begin{enumerate}
		\item For every $\seq$, $C_1$ is parallel to $F$ (that is, $C_1$ is in $L_F$, which is also the linear span of differences of points in $\mathcal{F}$), and
		\item $C_2$ approaches zero.
	\end{enumerate}
	
	For the first claim, notice that $C_1$ is a linear combination of vectors in $\mathcal{F}$, where the sum of the coefficients is
	$$\sum_{\mathbf{m} \in \mathcal{F}} c_{\mathbf{m}} \mathbf{z^{m-\mathfrak{v}}} + \mathbf{z^{-\mathfrak{v}}} (H - H_{\mathcal{F}}) = \mathbf{z^{-\mathfrak{v}}} H_{\mathcal{F}} + \mathbf{z^{-\mathfrak{v}}} (H - H_{\mathcal{F}}) = \mathbf{z^{-\mathfrak{v}}} H,$$
	which is zero for all $\seq$ because $\seq \in V$. Therefore, for all $\seq$, $C_1$ is in the linear span of differences of vectors in $\mathcal{F}$, so it is in $L_F$ (and parallel to $F$).
	
	For the second claim, we will need to apply Lemma \ref{monomials} to show which monomials approach finite values and which approach zero. By construction, we have that $\tilde{Q} = Q - \mathbf{v} = \kappa (\NP - \mathfrak{v}) = \kappa \tilde{\NP} \supseteq \tilde{\NP}$ (because $\tilde{\NP}$ has the origin as a vertex), and similarly, $\tilde{F} = \kappa \tilde{\mathcal{F}} \supseteq \tilde{\mathcal{F}}$. From this, we see from Lemma $\ref{monomials}$ that when $\mathbf{m} \in \mathcal{F}$ (or equivalently, $\mathbf{m} - \mathfrak{v} \in \tilde{\mathcal{F}} \subseteq \tilde{F}$), $\mathbf{z^{m-\mathfrak{v}}}$ approaches a finite nonzero number, and when $\mathbf{m} \in \NP \backslash \mathcal{F}$ (or equivalently, $\mathbf{m} - \mathfrak{v} \in \tilde{\NP} \backslash \tilde{\mathcal{F}} \subseteq \tilde{Q} \backslash \tilde{F}$), $\mathbf{z^{m-\mathfrak{v}}}$ approaches zero. However, because $H - H_{\mathcal{F}}$ contains only the terms of $H$ whose monomials are in $\NP \backslash \mathcal{F}$, $C_2$ is a sum of finitely many terms all of whose monomials are of the form $\mathbf{z^{m-\mathfrak{v}}}$ for $\mathbf{m} \in \NP \backslash \mathcal{F}$, so all of these terms (and $C_2$ itself) approach zero.
	
\end{proof}







We know that when $p$ is a point in the interior of the face at infinity corresponding to the face $F$ of $Q$, and $\mathbf{m}$ is a lattice point in the corresponding face $\mathcal{F}$ of $\NP$, the limiting value of $\mathbf{\seq^{m-\mathfrak{v}}}$ (where $\mathfrak{v}$ is a vertex of $\mathcal{F}$ as before) when $\seq$ approaches $p$ is nonzero and depends only on $p$, and not on the sequence $\sequ$ used to approach $p$. For the following theorem, it will be useful to define $p^{\mathbf{m-\mathfrak{v}}}$ to be that limiting value.

\begin{thm}
	\label{genericity}
	The hypotheses of Theorem \ref{genericresult} are generic; that is, if we fix a Newton polytope $\NP$, then for generic coefficients $c_{\mathbf{m}}$ of $H$, $\mathbf{z}^{-\mathfrak{v}} \nabla_{\log} H$ does not approach the zero vector at any point at infinity $p \in X_A$ that is the limit of $\Phi(\seq)$ for some sequence $\sequ$. (This means that, for a given Newton polytope $\NP$, the space of possible Laurent polynomials $H$ such that there exists a point $p$ at infinity in $X_A$ at which the normalized log-gradient under the standard rescaling converges to the zero vector, has positive codimension in the space of all Laurent polynomials with Newton polytope $\NP$.)
\end{thm}
\begin{proof}
	First, we will fix a specific $p$. We saw above in the proof of the second claim of Theorem \ref{genericresult} that $C_2$ and $\mathbf{z^{-\mathfrak{v}}} (H - H_{\mathcal{F}}) \mathfrak{v}$ approach zero, so as $\seq$ approaches $p$, we have that $\mathbf{z}^{-\mathfrak{v}} \nabla_{\log} H$ approaches the same vector as $\sum_{\mathbf{m} \in \mathcal{F}} c_{\mathbf{m}} \mathbf{z^{m-\mathfrak{v}} m}$, which is $\sum_{\mathbf{m} \in \mathcal{F}} c_{\mathbf{m}} p^{\mathbf{m-\mathfrak{v}}} \mathbf{m}$. Therefore, in order for $\mathbf{z}^{-\mathfrak{v}} \nabla_{\log} H$ to approach the zero vector, we must have that the coefficients of $H$ satisfy $$\sum_{\mathbf{m} \in \mathcal{F}} c_{\mathbf{m}} p^{\mathbf{m-\mathfrak{v}}} \mathbf{m} = 0.$$ Similarly, because $\sequ$ is contained in $\sing$, we have that for every $\seq$, $\mathbf{z}_n^{-\mathfrak{v}} H(\seq) = \sum_{\mathbf{m} \in \NP} c_{\mathbf{m}} \mathbf{z}^{\mathbf{m-\mathfrak{v}}} = 0$, and because $\mathbf{z^{m-\mathfrak{v}}}$ approaches zero when $\mathbf{m} \in \NP \backslash \mathcal{F}$, we must have that $$\sum_{\mathbf{m} \in \mathcal{F}} c_{\mathbf{m}} p^{\mathbf{m-\mathfrak{v}}} = 0$$ in order for $p$ to be a limit of a sequence contained in $\sing$. Combining these, we get $$\sum_{\mathbf{m} \in \mathcal{F}} c_{\mathbf{m}} p^{\mathbf{m-\mathfrak{v}}} \begin{bmatrix} \mathbf{m} \\ 1 \end{bmatrix} = 0.$$
	
	For this specific $p$, this condition means that the coefficient vector $\mathbf{c}$ of $H$ lies in the null space of the matrix $B_p$ whose columns are $p^{\mathbf{m-v}} \begin{bmatrix} \mathbf{m} \\ 1 \end{bmatrix}$. Let $s$ be the number of lattice points in $\mathcal{F}$ (which is the number of entries in $\mathbf{c}$), and let $j$ be the dimension of $\mathcal{F}$. Then because $\mathcal{F}$ has dimension $j$, there are $j+1$ affinely independent lattice points in $\mathcal{F}$, so there are $j+1$ linearly independent vectors among the $\begin{bmatrix} \mathbf{m} \\ 1 \end{bmatrix}$, and because all of the $p^{\mathbf{m-\mathfrak{v}}}$ are nonzero, $B_p$ has rank $j+1$, so its null space has dimension $s-(j+1)$. By Lemma \ref{faces}, the face at infinity corresponding to $F$ has dimension $j$, the same as the dimension of $F$, so the set of coefficient vectors $\mathbf{c}$ for which there exists such a $p$ for this face $F$ is the union of these null spaces and has dimension at most $s-(j+1) + j = s-1$. Finally, taking the union over the finitely many faces $F$ of $Q$ does not change the dimension, so the condition is generic.
\end{proof}

This genericity will turn out to be quite important. We will later see an example of a heighted CPAI in an unexpected direction (not parallel to any proper face of $Q$) in a case when the generic condition manages to fail at a point $p$ in the face at infinity corresponding to a face $F$ of codimension 2.

\subsection{For $H$ satisfying the generic condition, CPAI's can only occur in directions parallel to proper faces of $Q$ and must be heighted.}
We now show that convergence of the height function and convergence of the log-gradient should imply that all CPAI are heighted.

\begin{corollary}
	\label{genericheighted}
	Suppose that the rescaled log-gradient never vanishes on the toric compactification of the variety at any face at infinity. Then our previous results imply that any critical point at infinity on such a variety must be heighted and in a direction parallel to some face $F$ of the Newton polytope $\NP$ of $H$.
\end{corollary}
\begin{proof}
	Let $\sequ$ be a sequence witnessing a critical point at infinity, meaning that $\Phi(\seq)$ converges to some point $p \in X_A$, and $\grad_{\log} H$ converges projectively to some direction $R \in \mathbb{CP}^{d-1}$. Then $p$ lies in $X^0(F)$ for some face $F$, so by Theorem \ref{genericresult}, $R$ is parallel to $F$. Therefore, $p$ is a CPAI in only one projective direction, and it is parallel to $F$, so by Corollary \ref{heightconf} of Lemma \ref{monomials}, the height function in this direction converges to a finite number, so $p$ is a heighted CPAI.
\end{proof}

\section{Examples and counterexamples}

\subsection{Paraboloid example: $H = 1 - 2x + x^2 + 1 - 2y + y^2 - z$}


In this example, $\NP$ is the convex hull of (0,0,0), (2,0,0), (0,2,0), and (0,0,1), so $X_A$ is actually smooth everywhere except the point corresponding to the apex (0,0,1) in $\NP$. The variety given by $z = (x-1)^2 + (y-1)^2$ is a smooth paraboloid when viewed in affine space, but there is a point on the compactified variety such that the direction of the log-gradient vector cannot be defined uniquely. The Newton polytope is a tetrahedron whose base lies in the $xy$-plane, and the gradient of $H$ at (1,1,0) is orthogonal to the $xy$-plane. For a point $(x,y,z)$ on $V^*$, the log-gradient direction is given by

$$[-2x + 2x^2 : -2y + 2y^2 : -z] = [2x(x-1) : 2y(y-1) : -(x-1)^2 -(y-1)^2],$$

where the $z$-coordinate vanishes to higher order at (1,1,0) than the $x$- and $y$-coordinates (so that, at least in this example, any limiting log-gradient direction of a sequence of points approaching (1,1,0) will be parallel to the $xy$-plane). If we choose a sequence of points on $V^*$ with $x = 1$ and $y$ approaching 1, we get that the log-gradient direction is

$$[0 : 2y(y-1) : -(y-1)^2] = [0 : 2y : -(y-1)],$$

which approaches $[0 : 2 : 0]$. Similarly, if we choose a sequence of points with $y = 1$ and $x$ approaching 1, the log-gradient direction approaches $[2 : 0 : 0]$ (and any other projective direction parallel to the $xy$-plane can be achieved by choosing sequences approaching (1,1,0) with $\frac{x-1}{y-1}$ held constant). Furthermore, in either case, the points have images in $X_A$ that converge to the same point $p$ in the compactified variety, the point whose coordinates in $\mathbb{CP}^{s-1}$ are 1 in all positions corresponding to monomials in the $xy$-plane (that is, not containing the variable $z$), and 0 in the rest. Therefore, there is no way to extend the direction of the log-gradient continuously to this point $p \in X_A$.

\subsection{A heighted CPAI in an unexpected direction}
\label{unexpected}

Let $H = z - y - (x-1)^2$, and let $\mathcal{F}$ be the edge (one-dimensional face) determined by (0,0,0) and (1,0,0). Note that as in the previous two examples, we can choose $\mathfrak{v} = \mathbf{0}$. We can easily see that any sequence $\sequ$ in $(\mathbb{C}^*)^d$ that converges to (1,0,0) in $\mathbb{C}^d$ will have images $\Phi(\seq)$ in $X_A$ that converge to the point $p$ whose coordinates are $1$ for $\mathbf{m}$ in $F$ and $0$ otherwise. The curve defined by $(1 + t, t, t + t^2)$ is contained in the variety $\sing = V(H)$, and $(1+t, t, t + t^2)$ converges to (1,0,0) as $t \rightarrow 0$. However, $\nabla_{\log} H$ has direction

$$[-2x(x-1) :-y : z] = [-2(1+t)t : -t : t + t^2] = [-2(1+t) : -1 : 1 + t],$$

which converges to [-2 : -1 : 1]. This direction is not parallel to any face of the Newton polytope of $H$.

Furthermore, the exponentiated height in this direction is

$$ \frac{z}{x^2 y} = \frac{t + t^2}{(1 + t)^2 t} = \frac{1}{1 + t},$$

which converges to 1 as $t \rightarrow 0$. Therefore, this is a heighted CPAI in a direction that is not parallel to any face of the Newton polytope! Despite this, and despite the fact that the height function in this direction does not extend continuously to $\mathcal{F}$, the height function does approach finite values along certain curves in $V^*$.  As a matter of fact, in this example we can use the fact that $z = y + (x-1)^2$ to substitute different functions of $t$ for $x$ and $y$ and find parameterizations for curves approaching (1,0,0) for which the log-gradient of $H$ approaches any direction at all of the form $[\alpha:-1:1]$. For example, if we take the path 
$$(1 + \gamma t, t, t + \gamma^2 t^2) \in \sing,$$ 
we get that $\nabla_{\log} H$ has direction $$[-2x(x-1) : -y : z] = [-2(1 + \gamma t) \gamma t : -t : t + \gamma^2 t^2] = [-2 \gamma (1 + \gamma t) : -1 : 1 + \gamma^2 t],$$ which converges to $[-2 \gamma : -1 : 1]$. On the other hand, we can instead take the path 
$$(1 + t, t^2, 2t^2) \in \sing,$$ 
which results in a log-gradient direction of 
$$[-2x(x-1) :-y : z] = [-2(1+t)t : -t^2 : 2t^2] = [-2(1+t) : -t : 2t],$$ 
which converges to the direction [-2:0:0] as $t \to 0$ and is indeed parallel to $F$. 

\section{Beyond the Generic Case}

When the generic hypotheses of Theorem \ref{genericresult} are satisfied, any point $p$ at infinity in the compactification of $\sing$ in $X_A$ has a unique direction to which the log-gradient can converge when a sequence $\sequ$ in $(\mathbb{C}^*)^d$ has images $\Phi(\seq)$ in $X_A$ converging to $p$, and that direction is always parallel to the face $F$ for which $p$ is in the corresponding face at infinity. We now turn our attention beyond the generic condition mentioned above and look for information about limiting log-gradient directions for a point $p$ in some cases in which the set of directions to which log-gradients can converge as $\Phi(\seq)$ approaches $p$ has positive dimension. But first, some background concerning monomial transformations is in order.

\subsection{Background: Monomial Transformations}


Let $A$ be a $d$-by-$d$ matrix of integers that is invertible over the rationals. For our purposes, the \textbf{monomial transformation} defined by $A$ is the map $\tau_A: (\mathbb{C}^*)^d \rightarrow (\mathbb{C}^*)^d$ given by $$\tau_A(z_1, ..., z_d) = \left(\prod_{j=1}^d (z_j^{a_{1j}}), ..., \prod_{j=1}^d (z_j^{a_{dj}})\right).$$ Each coordinate is a Laurent monomial and may contain negative exponents. Notice that when $\tau_A$ is defined in this way, we have that $\tau_A(\mathbf{z}) = e^{A\log(\mathbf{z})}$ (which is single-valued because $A$ is an integer matrix, even though $\log(\mathbf{z})$ is multivalued), where $e^{\mathbf{z}}$ and $\log(\mathbf{z})$ refer to componentwise exponentiation and natural logarithm respectively (and where we consider $\mathbf{z}$ as a column vector). Therefore, monomial transformations correspond to integer linear transformations in log space (the space where the coordinates are the natural logarithms of the original coordinates), with the caution that $\tau_A$ is injective only if $A$ is unimodular (has determinant $\pm 1$). However, $\tau_A$ is surjective if $A$ is invertible, even if it is not unimodular: one point in the preimage of $\mathbf{z}$ can be found as $e^{A^{-1} \operatorname{Log} (\mathbf{z})}$, where $\operatorname{Log}$ takes the principal value of the natural logarithm of each component


A monomial transformation $\tau_A$ naturally induces a map on Laurent polynomials $H$ in $\mathbb{C}[z_1, z_1^{-1}, ..., z_d, z_d^{-1}]$, given by $\tau_A^*(H) = H \circ \tau_A$. For example, if $H(\mathbf{z}) = z_i$ simply takes the $i^{\mathrm{th}}$ component, then $\tau_A^*(H) = \prod_{j=1}^d z_j^{a_{ij}}$, and therefore if $H(\mathbf{z}) = \mathbf{z^r}$ is any Laurent monomial then $$\tau_A^*(H) = \prod_{j=1}^d z_j^{\sum_{i=1}^d r_i a_{ij}} = \mathbf{z}^{A^T \mathbf{r}}.$$ We see that performing a monomial transformation will perform a linear transformation on the exponent vector of each term of $H$, and therefore on the Newton polytope $\NP$, so the following basic lemma about linear tranformations on polytopes will be useful.

\begin{lem}
	Let $\NP$ be the convex hull of a finite set $\mathcal{A} \subseteq \mathbb{R}^d$, and let $T$ be an affine transformation (a linear transformation $L$ followed by a translation) on $\mathbb{R}^d$. Then the convex hull of $T(\mathcal{A})$ is $T(\NP)$. Furthermore, if $T$ is invertible, then for every face $\mathcal{F}$ of $\NP$, we have that $T(\mathcal{F})$ is a face of $T(\NP)$.
\end{lem}
\begin{proof}
	The first statement follows directly from linearity of $L$ and the definition of convex hull. For the second, if $S$ is a supporting hyperplane for $\mathcal{F}$ as a face of $\NP$, then injectivity implies that the preimage of $T(S) \cap T(\NP)$ is no larger than $S \cap \NP = \mathcal{F}$, so $T(S)$ is a supporting hyperplane for $T(\NP)$.
\end{proof}

It is a well-known fact that if $g(\mathbf{z}) = f(A\mathbf{z})$, then $\nabla g(\mathbf{z}) = A^T ((\nabla f)(A\mathbf{z}))$. It is not too difficult to come up with a similar equation for log-gradients and monomial transformations.

\begin{lem}
	\label{monomchain}
	Let $\tau_A$ be the monomial transformation defined by the $d$-by-$d$ matrix $A$, let $f: \mathbb{C}^d \rightarrow \mathbb{C}$ be a function that is continuously differentiable in a neighborhood of $\tau_A(\mathbf{z})$, and let $g(\mathbf{z}) = f(\tau_A(\mathbf{z}))$. Then $$\grad_{\log} g (\mathbf{z}) = A^T \grad_{\log} f(\tau_A(\mathbf{z})),$$ where $\grad_{\log} f(\tau_A(\mathbf{z}))$ means the log-gradient of the outer function $f$ evaluated at the inner function $\tau_A(\mathbf{z})$.
\end{lem}
\begin{proof}
	If $\mathbf{y} = \tau_A(\mathbf{z})$ and $g(\mathbf{z}) = f(\tau_A(\mathbf{z})) = f(\mathbf{y})$, then we have by the multivariate chain rule that
	
	$$ \nabla_{\log} g = \begin{bmatrix} z_1 \frac{\partial g}{\partial z_1} \\ \vdots \\ z_d \frac{\partial g}{\partial z_d} \end{bmatrix} = \begin{bmatrix} z_1 (\frac{\partial f}{\partial y_1} \cdot \frac{\partial}{\partial z_1} \prod_{j=1}^d (z_j^{a_{1j}}) + ... + \frac{\partial f}{\partial y_d} \cdot \frac{\partial}{\partial z_1} \prod_{j=1}^d (z_j^{a_{dj}})) \\ \vdots \\ z_d (\frac{\partial f}{\partial y_1} \cdot \frac{\partial}{\partial z_d} \prod_{j=1}^d (z_j^{a_{1j}}) + ... + \frac{\partial f}{\partial y_d} \cdot \frac{\partial}{\partial z_d} \prod_{j=1}^d (z_j^{a_{dj}})) \end{bmatrix} $$ 
	
	$$ = \begin{bmatrix} a_{11} \frac{\partial f}{\partial y_1} \prod_{j=1}^d (z_j^{a_{1j}}) + ... + a_{d1} \frac{\partial f}{\partial y_d} \prod_{j=1}^d (z_j^{a_{dj}}) \\ \vdots \\ a_{1d} \frac{\partial f}{\partial y_1} \prod_{j=1}^d (z_j^{a_{1j}}) + ... + a_{dd} \frac{\partial f}{\partial y_d} \prod_{j=1}^d (z_j^{a_{dj}}) \end{bmatrix} = A^T \begin{bmatrix} \frac{\partial f}{\partial y_1} \prod_{j=1}^d (z_j^{a_{1j}}) \\ \vdots \\ \frac{\partial f}{\partial y_d} \prod_{j=1}^d (z_j^{a_{dj}}) \end{bmatrix},$$
	
	where each $\frac{\partial f}{\partial y_i}$ means the value of that partial derivative at $\tau_A(\mathbf{z})$. The value of the monomial $z_i$ at $\tau_A(\mathbf{z})$ is $\prod_{j=1}^d (z_j^{a_{ij}})$, so the result follows.
\end{proof}

\begin{remark}
	For the case where $F$ is a Laurent polynomial, there is a shorter proof. If $f(\mathbf{z}) = \mathbf{z^r}$ is a Laurent monomial, then $\grad_{\log} f(\mathbf{z}) = \mathbf{r z^r} = \mathbf{r} f(\mathbf{z})$, and $g(\mathbf{z}) = f(\tau_A(\mathbf{z})) = \mathbf{z}^{A^T \mathbf{r}}$, so $$\grad_{\log} g(\mathbf{z}) = A^T \mathbf{r} \mathbf{z}^{A^T \mathbf{r}} = A^T \mathbf{r} f(\tau_A(\mathbf{z})) = A^T \grad_{\log} f(\tau_A(z)),$$ and the result for all Laurent polynomials then follows by linearity.
\end{remark}

\subsection{Transforming the Newton Polytope for Analysis Near $X^0(F)$}
\label{transfpoly}

The goal of this section is to transform coordinates so that a face at infinity becomes an intersection of coordinate hyperplanes $\{ y_1 = \cdots = y_k = 0 \}$. Subject to certain conditions, it will then be possible to compute the space of possible limiting log-gradient directions, and therefore the set of all possible directions for CPAI. 

Recall that when the rescaled log-gradient approaches a nonzero vector, a point $p \in X_A$ can be a CPAI for a single direction only. However, when it fails (typically at only finitely many points), it will be possible for a single point to be a CPAI for a codimension-1 set of directions. When $F$ is not a facet, not all of these will be parallel to $F$.

Let $\mathcal{F}$ be a facet (codimension-1 face) of the Newton polytope $\NP$. If the $d$ lattice points $\mathbf{v}, \mathbf{v}_1, ..., \mathbf{v}_{d-1} \in \mathbb{Z}^d$ are affinely independent, they define a hyperplane normal to the integer vector given by the formal $d$-by-$d$ determinant $\left| \begin{matrix} \mathbf{e}_1 & \cdots & \mathbf{e}_d \\ & \mathbf{v}_1 - \mathbf{v} \\ & \vdots & \\ & \mathbf{v}_{d-1} - \mathbf{v} \end{matrix} \right|$ of the ``matrix" whose first row has as entries the standard basis vectors, and whose second through last rows are the vectors $\vv_j - \vv$ (similar to taking a cross product). 

Define the \textbf{inward-pointing normal} to $\mathcal{F}$ to be the minimum-modulus integer vector $\mathbf{n}_{\mathcal{F}}$ that is normal to the facet $\mathcal{F}$ and is oriented such that $\mathbf{n}_{\mathcal{F}} \cdot \mathbf{m} \geq 0$ for all $\mathbf{m} \in \tilde{\NP} = \NP - \mathbf{v}$.  

\begin{pr}
	When $\mathcal{F}$ is a facet, it is possible to apply an affine transformation to $\NP$ to move the hyperplane defined by $\mathcal{F}$ to one of the coordinate hyperplanes, in such a way that the rest of the polytope ends up being above (as opposed to below) the coordinate hyperplane. If $A^T$ is the matrix of this transformation, then $\overline{H} = \tau_A^*(H/\mathbf{z}^{\mathfrak{v}})$ will have a constant term, but no negative powers of $z_d$.
\end{pr}
\begin{proof}
	We simply subtract $\mathfrak{v}$ and then left-multiply by any invertible (not necessarily orthogonal!) matrix $A^T$ whose last row is the inward-pointing normal. Under any such transformation, the lattice point $\mathbf{m}$ will be sent to a vector whose last entry is $\mathbf{n}_{\mathcal{F}} \cdot (\mathbf{m} - \mathfrak{v})$, which is nonnegative if $\mathbf{m} \in \NP$ and zero if $\mathbf{m} \in \mathcal{F}$. Therefore, $A^T \tilde{\mathcal{F}}$ will be contained in the coordinate hyperplane $z_d = 0$, and the transformed Laurent polynomial $\tau_A^*(H/\mathbf{z}^{\mathfrak{v}})$ will have no negative powers of $z_d$, but it may of course have negative powers of the other variables. The term that used to have exponent vector $\mathfrak{v}$ will become a constant term.
\end{proof}

If the sequence converges to the face at infinity corresponding to $F$ before the monomial transformation, then it (or any subsequence of it) can only converge to the face at infinity corresponding to $\{z_d = 0\}$ afterwards; assuming we chose the columns of $A$ to be a lattice basis, which is possible in the case of a facet, the same set of coordinates in $\mathbb{CP}^{s-1}$ is approaching zero.

Now suppose that $\mathcal{F}$ is a codimension-$k$ face of $\NP$ with $k > 1$, and suppose that $\mathcal{F}_1, ..., \mathcal{F}_k$ are $k$ facets whose intersection is $\mathcal{F}$ with linearly independent inward-pointing normals. In the remainder of this section, we will apply similar logic to construct an invertible monomial transformation matrix (which will be called $N$ from this point forward to avoid confusion with the set $A$ of lattice points defining $X_A$) with the last $k$ columns being the inward-pointing normals to these $k$ facets, and $N^T \tilde{\NP}$ will be contained in $\mathbb{R}^{d-k} \times \mathbb{R}_{\geq 0}^k$ (the intersection of half-spaces where the last $k$ variables are nonnegative).

\begin{definition}
	\label{montransf}
	Suppose that $\mathcal{F}$ is a codimension-$k$ face of the Newton polytope $\NP$ of a Laurent polynomial $H$, let $\mathfrak{v}$ be a vertex of $\FF$, and let $\mathbf{v} = \kappa \mathfrak{v}$ be the corresponding vertex of the face $F = \kappa \FF$ of $Q$. Suppose that $\mathcal{F}_1, ..., \mathcal{F}_k$ are $k$ facets whose intersection is $\mathcal{F}$ with linearly independent inward-pointing normals, and let $F_1, ..., F_k$ be the corresponding facets of $Q$. Let $N$ be any integer matrix, invertible over the rationals, whose last $k$ columns are the inward-pointing normals to $\mathcal{F}_1, ..., \mathcal{F}_k$. We define the following notation:
	\begin{enumerate}
		\item $\overline{H} = \tau_N^*(\mathbf{z}^{-\mathfrak{v}} H)$,
		\item $\overline{\NP} = N^T(\NP - \mathfrak{v})$, the Newton polytope of $\overline{H}$,
		\item $\overline{Q} = \kappa \overline{\NP}$, an integer multiple of $\overline{\NP}$ that is normal (any integer $\kappa \geq \max\{1, d-1\}$ suffices, so we can assume WLOG that the same $\kappa$ is used to normalize both $Q$ and $\overline{Q}$),
		\item $\overline{A} = \overline{Q} \cap \mathbb{Z}^d = \{\overline{\mm}_1, ..., \overline{\mm}_{\overline{s}}\}$ (in general, we can have $\overline{s} \geq s$), and
		\item $\overline{\Phi}: (\mathbb{C}^*)^d \rightarrow \mathbb{CP}^{\overline{s}-1}$ is the map given by $\overline{\Phi}(\mathbf{w}) = [\mathbf{w}^{\overline{\mm}_1}, ..., \mathbf{w}^{\overline{\mm}_{\overline{s}}}]$ (and $X_{\overline{A}}$ is the closure of the image of the map $\overline{\Phi}$).
	\end{enumerate}
\end{definition}

\begin{remark}
	We cannot say that $\overline{\Phi} = \Phi \circ \tau_N$: their target spaces can be in different projective spaces, as when $\overline{Q}$ contains additional lattice points that are not images of lattice points of $\tilde{Q}$ under the linear transformation given by $N^T$. Also, $\overline{\Phi}$ is injective, while $\Phi \circ \tau_N$ may not be.
\end{remark}

The following lemma about monomial transformations helps us to reduce to the case where $p$ is located in a face lying in the intersection of the last $k$ coordinate hyperplanes.

\begin{lem}
	\label{seqtransf}
	Under the assumptions and notation in Definition \ref{montransf}, suppose also that $\sequ \subseteq \sing^*$ is a sequence such that $\Phi(\seq)$ converges to some point $p$ in the face at infinity in $X_A$ corresponding to $F$, and that the projective direction of $\grad_{\log} H$ evaluated at $\seq$ converges in $\mathbb{CP}^{d-1}$ to some direction $R \in \mathbb{CP}^{d-1}$ (but the magnitude of $\grad_{\log} H$ might approach zero or infinity). Then:
	\begin{enumerate}
		\item $\overline{H}$ has a constant term, but no negative powers of any of the last $k$ variables, and
		\item There exists a sequence $\sequw \subseteq V(\overline{H}) \cap (\mathbb{C}^*)^d$ such that: 
		\begin{enumerate}
			\item $\left\{\tau_N(\seqw)\right\}_{n=1}^{\infty}$ is a subsequence of $\sequ$,
			\item The images of $\seqw$ under the map $\overline{\Phi}$ converge in $X_{\overline{A}}$ to a point $\overline{p}$ on the face at infinity corresponding to the face $\overline{F} = N^T\tilde{F}$ of codimension $k$ that is contained in the intersection of the last $k$ coordinate hyperplanes, and
			\item $\grad_{\log} \overline{H}$ evaluated at $\seqw$ converges projectively to $N^T R$, which is parallel to the intersection of the last $k$ coordinate hyperplanes if and only if $R$ is parallel to $F$.
		\end{enumerate}
	\end{enumerate}
\end{lem}
\begin{proof}
	For item 1, note that the exponent vectors of the monomials in $\overline{H}$ are exactly those of the form $N^T(\mathbf{m} - \mathfrak{v})$, where $\mathbf{m}$ is a monomial vector in $H$. We know that $H$ has a term with exponent vector $\mathfrak{v}$ because $\mathfrak{v}$ is a vertex of the Newton polytope of $H$, so $\overline{H}$ has a term with monomial $N^T(\mathfrak{v} - \mathfrak{v}) = \mathbf{0}$ (a constant term). Furthermore, for each $\FF_j$, the inward-pointing normal $\mathbf{n}_{\FF_j}$ has the property that $\mathbf{n}_{\FF_j} \cdot \mathbf{\tilde{m}} \geq 0$ for all $\mathbf{\tilde{m}} \in \tilde{\NP}$, and the last $k$ coordinates of any exponent vector of $\overline{H}$ (which are $N^T(\mathbf{m} - \mathfrak{v})$ for $\mathbf{m}$ an exponent vector of a term in $H$, so in particular $\mathbf{m} \in \NP$) are all of the form $\mathbf{n}_{\FF_j} \cdot \mathbf{\tilde{m}}$ for some $\mathbf{\tilde{m}} \in \tilde{\NP}$ and therefore nonnegative.
	
	For item 2(a), first recall that $\tau_N$ is surjective, so for each $n$ we will let $\mathbf{u}_n$ be a point in the preimage of $\mathbf{z}_n$ under $\tau_N$, and by compactness of $X_{\overline{A}}$ as a closed subset of projective space, this sequence has a subsequence $\sequw$ (with $\tau_N(\mathbf{w}_n) = \mathbf{z}_{m_n}$ for some increasing sequence of positive integers $\left\{m_n\right\}_{n=1}^{\infty}$) such that $\overline{\Phi}(\seqw)$ converges in $X_{\overline{A}}$ to some point $\overline{p}$. Clearly, $\overline{H}(\mathbf{w}_n) = \mathbf{z}_{m_n}^{-\mathfrak{v}} H(\zz_{m_n}) = 0$ because $\zz_{m_n} \in \sing^*$.
	
	For 2(b), it is clear that $\overline{F}$ has codimension $k$, and it is clear that it is contained in the intersection of the last $k$ coordinate hyperplanes because each of the last $k$ coordinates of a point in $\overline{F} = N^T(F - \mathbf{v})$ is the dot product of a normal to a facet $F_j$ with a vector that is a difference of two points in $F \subseteq F_j$, and is therefore zero. We know by Lemma \ref{faces} that $\overline{p}$ is in the face at infinity in $X_{\overline{A}}$ corresponding to some face $\overline{G}$ of $\overline{Q}$; we need to show that $\overline{G} = \overline{F}$. Notice that vertices $\mm$ of $Q$ map in one-to-one correspondence to vertices $N^T(\mathbf{m} - \mathbf{v})$ of $\overline{Q}$ (the same cannot be said in general with ``vertices" replaced by ``lattice points"), with $\mm \in F$ if and only if $N^T(\mathbf{m} - \mathbf{v}) \in \overline{F}$. Notice also that if $\mathbf{m}$ is a lattice point of $Q$, then the component of $\overline{\Phi}(\mathbf{w}_n)$ corresponding to the lattice point $N^T(\mathbf{m} - \mathbf{v}) \in \overline{Q}$ is $\mathbf{w}_n^{N^T(\mathbf{m} - \mathbf{v})} = \mathbf{z}_{m_n}^{\mathbf{m} - \mathbf{v}}$, a monomial that converges to a finite value that is nonzero if $\mathbf{m} \in F$, and zero if $\mathbf{m} \in Q \backslash F$, by Lemma \ref{monomials} because $p \in X^0(F)$. (Recall that the origin is in $\overline{Q}$, so no further rescaling is necessary in $\overline{\Phi}$ because one of the coordinates of $\overline{\Phi}$ is identically 1, and all coordinates approach finite values.) Therefore, the face $\overline{G}$ contains $N^T(\mathbf{m} - \mathbf{v})$ for all $\mathbf{m} \in F$, and because it is convex, it contains their convex hull, which is $\overline{F}$. On the other hand, if $\overline{G}$ is not contained in $\overline{F}$, then $\overline{G}$ contains a vertex of $\overline{Q}$ that is not in $\overline{F}$, and this vertex is of the form $N^T(\mathbf{m} - \mathbf{v})$ for some vertex (in particular, some lattice point) $\mathbf{m}$ of $Q$ that is not in $F$, so that $\mathbf{w}_n^{N^T(\mathbf{m} - \mathbf{v})} = \mathbf{z}_{m_n}^{\mathbf{m} - \mathbf{v}}$ approaches zero (contradiction).
	
	For 2(c), for ease of notation we let $\tilde{H} = \zz^{-\mathfrak{v}} H$. We first note by Lemma \ref{monomgradlog} that $\grad_{\log} \tilde{H}(\zz_{m_n})$ is projectively equivalent to $\grad_{\log} H(\zz_{m_n})$ because $\zz_{m_n} \in \sing^*$. Because $\overline{H} = \tilde{H} \circ \tau_N$, we have by Lemma \ref{monomchain} that $\grad_{\log} \overline{H}(\mathbf{w}_n) = N^T \grad_{\log} \tilde{H} (\zz_{m_n})$. Multiplication by an invertible matrix $N^T$ is still well-defined as a continuous function from $\mathbb{CP}^{d-1}$ to itself (because $N^T(c\mathbf{r}) = c N^T \mathbf{r}$, and because $N^T$ is invertible we have that $N^T \mathbf{r} = 0$ implies $\mathbf{r} = 0$), so if $p$ is the standard projection map from $\mathbb{C}^* \backslash \{\mathbf{0}\}$ to $\mathbb{CP}^{d-1}$ (the assumption that $\grad_{\log} H (\zz_n)$ converges projectively to anything implies that its value is not the zero vector for sufficiently large $n$, even though its magnitude may converge to zero), then the equation $$p(\grad_{\log} \overline{H}(\mathbf{w}_n)) = N^T p(\grad_{\log} H (\zz_{m_n}))$$ holds in projective space for all sufficiently large $n$. By assumption, $p(\grad_{\log} H (\zz_{m_n}))$ converges in $\mathbb{CP}^{d-1}$ to $R$, so $p(\grad_{\log} \overline{H}(\mathbf{w}_n))$ converges in $\mathbb{CP}^{d-1}$ to $N^T R$ as desired (even though $\grad_{\log} \overline{H}(\mathbf{w}_n)$ may itself have converged to the zero vector in $\mathbb{C}^d$).
\end{proof}


\subsection{A Modified Simple Condition}
\label{secsimple}



The goal of the next few sections is to determine the possible directions in which a CPAI can occur for some meaningful examples when the generic condition in Theorem~\ref{genericity} fails somewhere in the closure of $\sing^*$ in $X_A$. 
The result ultimately achieved, Theorem~\ref{simplecaseresult}, holds when $p$ lies on a face $F$ such that $\sigma_F/L_F$ is simplicial, the variety $V(\overline{H})$ is smooth near $p$ after a monomial transformation, and the Jacobian of the log-gradient of $\overline{H}$ is nondegenerate. The conclusion is that a single point can be a critical point at infinity for a set of directions of codimension~1 but not a set of directions of full dimension. When $p$ is in a face $F$ of codimension~2 or more, then $p$ can be a CPAI for directions not parallel to $F$.

We have to be careful because the geometry of the polytope near a face of codimension $k$ may not be the same as the product of a $k$-dimensional orthant with $\R^{d-k}$. As an example of what happens when we apply a monomial transformation in such a case, the apex of a square pyramid is contained in four facets, so near its apex, a square pyramid is not diffeomorphic to an orthant in $\R^3$, and no monomial coordinate change can make each of the facets containing the apex lie in a coordinate hyperplane. Let $H = 1 + x + y + xy + z$. Then $Q$ is a square pyramid, and a factor of 2 is sufficient to make $Q$ normal. If $F$ is the apex, $F$ is the intersection of four facets, with normals $\begin{bmatrix} 1 \\ 0 \\ 0 \end{bmatrix}$, $\begin{bmatrix} 0 \\ 1 \\ 0 \end{bmatrix}$, $\begin{bmatrix} -1 \\ 0 \\ -1 \end{bmatrix}$, and $\begin{bmatrix} 0 \\ -1 \\ -1 \end{bmatrix}$. To construct a monomial transformation, we have to choose three of these four normal vectors to be able to transform them into the three coordinate vectors. This monomial transformation can be constructed by placing these vectors into the rows of $N^T = \begin{bmatrix} 1 & 0 & 0 \\ 0 & 1 & 0 \\ -1 & 0 & -1 \end{bmatrix}$. The resulting transformation $\tau_N^*$ on $\mathbb{C}[x,x^{-1},y,y^{-1},z,z^{-1}]$ is given by $\tau_N^*(x) = xz^{-1}$, $\tau_N^*(y) = y$, and $\tau_N^*(z) = z^{-1}$. We divide by $z$ and apply the monomial transformation: $\overline{H} = \tau_N^*(H/z) = z + x + yz + xy + 1$ has a constant term and no negative powers (at all, because $F$ has dimension zero). However, $\overline{H}$ has no term involving $y$ alone, meaning that $y$ is not guaranteed to approach zero in a sequence $\sequw \subseteq (\mathbb{C}^*)^d$ whose images in $X_{\overline{A}}$ converge to the apex.

	If the polytope $Q$ is \textit{simple}, meaning that every vertex of $Q$ is adjacent to precisely $d$ edges (and no more), then the problem goes away immediately. If we let $\vv$ be a vertex of a face $F$, then the nonnegative span of $\tilde{Q} = Q - \vv$ is a simplicial cone $\sigma_{\vv}$ (the cone over a $(d-1)$-simplex, the only ($d-1$)-dimensional polytope with only $d$ vertices), and the faces of $Q$ containing $\vv$ correspond to faces of $\sigma_{\vv}$, which are the nonnegative linear spans of all subsets of the $d$ linearly independent vectors that generate $\sigma_{\vv}$. In particular, any face $F$ of codimension $k$ is the nonnegative linear span of some ($d-k$)-element subset of these vectors, and the facets containing $F$ are in correspondence with the ($d-1$)-element subsets of the generators of $\sigma_{\vv}$ containing the $d-k$ generators of $F$, of which there are exactly $k$.
	
	
	Rather than restrict to simple polytopes, the following definition allows us to conclude all geometric facts related to simple polytopes that we need for our purposes.
	\begin{defn}[modified simple condition]
		Say that $Q$ satisfies the {\em modified simple condition} if $\sigma_F/L_F$ is a cone over a ($k-1$)-simplex. This is the same thing as requiring that $F$ is contained in exactly $k$ (and no more) faces of $Q$ of dimension $1 + \dim F$.  
	\end{defn}
	
	Lemma \ref{facecorresp} below shows that for any convex polytope $Q$, we can take a small neighborhood around a point $q$ in the interior of $F$, intersect it with $Q$, and the cone we get from that is $\sigma_F$ and has faces $G'$ that correspond exactly to $Q$'s faces $G$ of the same dimension that contain $F$.
	
	\begin{lem}
		\label{facecorresp}
		Let $F$ be a face of $Q$ of codimension $k$. Then $\sigma_F$ is the nonnegative linear span of $\tilde{Q} - q$, where $q$ is a point in the interior of $\tilde{F}$, and $\sigma_F$ is also the intersection of the defining half-spaces of $\tilde{Q}$ whose bounding hyperplanes contain $\tilde{F}$.  Consequently, the ($d-k+r$)-dimensional faces of the cone $\sigma_F$ (or alternatively, the $r$-dimensional faces of the pointed cone $\sigma_F/L_F$) are in one-to-one-correspondence with the ($d-k+r$)-dimensional faces of $\tilde{Q}$ that contain $\tilde{F}$ (or alternatively, with the ($d-k+r$)-dimensional faces of $Q$ that contain $F$).
	\end{lem}
	\begin{proof}
		Let $q$ be a point in the interior of $\tilde{F}$, and let $N$ be a neighborhood of $q$ sufficiently small such that, out of all the hyperplanes that define the convex polytope $\tilde{Q}$, the only ones that intersect $N$ are those containing $\tilde{F}$. Then the nonnegative linear span of $(N \cap \tilde{Q}) - q$ (subtracting $q$ from every point, not set minus) is the cone $\tilde{\sigma}_F$ given by the intersection of the defining half-spaces of $\tilde{Q}$ whose bounding hyperplanes contain $\tilde{F}$ (because none of the other defining hyperplanes intersect $N$). We now have that the nonnegative linear span of $\tilde{Q}-q$ both contains $\tilde{\sigma}_F$ (the nonnegative linear span of $(N \cap \tilde{Q}) - q$) and is contained in $\sigma_F$ (the sum of $L_F$ and the nonnegative linear span of $\tilde{Q}$), while simultaneously, $\tilde{Q}$ and $L_F$ (and therefore also $\sigma_F$) are contained in the intersection of the defining half-spaces of $\tilde{Q}$ whose bounding hyperplanes contain $\tilde{F}$, which is $\tilde{\sigma}_F$. Therefore, $\tilde{\sigma}_F = \sigma_F$.
		
		Similarly, the corresponding face of $\sigma_F$ to a face $G$ of $\tilde{Q}$ containing $\tilde{F}$ is the nonnegative linear span of $(N \cap G) - q$, which is the sum of $L_F$ and the nonnegative linear span of $G$, which has the same dimension as $G$ and is also the face of $\sigma_F$ cut out by the same set of equations and inequalities that cut out $G$ as a face of $\tilde{Q}$ (except for the inequalities whose hyperplanes do not contain $\tilde{F}$). A face $G'$ of $\sigma_F$ corresponds to its intersection with $\tilde{Q}$, the face of $\tilde{Q}$ cut out by the same set of equations and inequalities that cut out $G'$ as a face of $\sigma_F$ (in addition to all of the defining inequalities of $\tilde{Q}$ whose hyperplanes do not contain $\tilde{F}$). Note that a supporting hyperplane for $G$ as a face of $\tilde{Q}$ is also a supporting hyperplane for $G'$ as a face of $\sigma_F$, and vice versa.
	\end{proof}
	
	Now we can see how the hypotheses that $Q$ is modified simple helps us transform $H$ into a Laurent polynomial $\overline{H}$ in a ``nice" form.
	
	\begin{lem}
		\label{modsimple}
		In addition to the hypotheses of Lemma \ref{seqtransf}, suppose that $\sigma_F/L_F$ is simplicial. Then $F$ is the intersection of precisely $k$ facets $F_1, ..., F_k$ and no proper subset of them. For convenience, denote the $d$ variables of $\overline{H}$ by $x_1, ..., x_{d-k}, y_1, ..., y_k$. Then for a suitable choice of $N$, $\overline{H}$ has a term that involves $y_j$ to a strictly positive power (and possibly also some subset of the $x$ variables, possibly to negative powers) but does not involve $y_{j'}$ for any $j' \neq j$.
	\end{lem}
	
	\begin{proof}
		If the $k$-dimensional pointed cone $\sigma_F/L_F$ is simplicial (viewed as being within the orthogonal complement of $L_F$ and generated by $k$ linearly independent vectors $\vv_1, ..., \vv_k$), then $\sigma_F/L_F$ has precisely $k$ facets (each facet $F_j$ is the set of all nonnegative linear combinations of the ($k-1$)-element subset of $\{\vv_1, ..., \vv_k\}$ that does not contain $\vv_j$), so by Lemma \ref{facecorresp}, $Q$ has precisely $k$ facets $F_1, ..., F_k$ containing $F$. The facet normals (within $L_F^{\perp}$) of the $k$ facets of $\sigma_F/L_F$ (which are the same as the facet normals of the $k$ facets of $\tilde{Q}$ containing $\tilde{F}$) are linearly independent (because if there were a linear relation $\sum_{j=1}^k c_j \mathbf{n}_{F_j}$ with some $c_j$ nonzero, then taking the dot product with the generator $\vv_j$ that is in every $F_{j'}$ except for $F_j$, would yield the contradiction that $\vv_j$ is also in $F_j$). Then $\cap_{j=1}^k \tilde{F}_j$ contains $\tilde{F}$, and because it is $(d-k)$-dimensional, it is contained in $L_F \cap \tilde{Q}$, which is just $\tilde{F}$ because $\tilde{F}$ is a face of $\tilde{Q}$. Therefore, $N$ can be chosen to be an invertible matrix whose last $k$ columns are the normals to the $k$ facets whose intersection is $F$.
		
		To show that every proper subset of the $F_j$ has intersection properly containing $F$, it suffices to find a point that is in $F_{j'}$ for every $j' \neq j$ but is not in $F_j$. This need not be a lattice point, so as in the proof of Lemma \ref{facecorresp}, we can simply take a sufficiently small neighborhood $U$ around a point $q$ in the interior of $\tilde{F}$, and $U \cap \tilde{Q}$ contains a point of the form $q + \epsilon \vv_j$ that is in every $F_{j'}$ except $F_j$.
		
		We now have that $\cap_{j' \neq j} F_{j'}$ is a face that properly contains $F$, so there exists a \textit{vertex} $\mathbf{u}_j$ of $\tilde{\NP}$ that is contained in every $\tilde{\FF}_{j'}$ except $\tilde{\FF}_j$ (where $\tilde{\FF}_j = \tilde{\FF} - \mathfrak{v}$ as usual). Then the term of $\mathbf{z}^{-\mathfrak{v}}H$ with exponent vector $\mathbf{u}_j$ is mapped by $\tau_N^*$ to a term of $\overline{H}$ with exponent vector $N^T \mathbf{u}_j$, which has $y_{j'} = 0$ for all $j' \neq j$, and $y_j > 0$ because $\mathbf{u}_j \not\in F_j$.
	\end{proof}
	
	This seems like a stringent condition, but it is actually much better than it sounds: because every pointed convex cone in 1 or 2 dimensions is simplicial, every face of codimension 1 or 2 (even when $d$ is large) will satisfy this modified simple condition! A sequence in $\sing^*$ cannot converge to a vertex in $X_A$ (because when approaching the point in $X_A$ corresponding to a vertex $\mathbf{v}$, the monomial $\zz^{\mathfrak{v}}$ has a nonzero coefficient and dominates all the other monomials of $H$), so all three-dimensional examples will satisfy this condition at every face $F$ to which a sequence in $\sing^*$ could converge. The first time that it can actually affect asymptotics is in four dimensions, in cases where the generic condition manages to fail at a point on a face of codimension 3.
	
	\subsection{Under certain conditions, $p$ can be a CPAI for a codimension-1 set of directions, but not a set of full dimension.}
	
	Now, we would like to find the set of directions in which a point $p \in X^0(F)$ can be a CPAI whenever $F$ satisfies the modified simple condition that $\sigma_F/L_F$ is simplicial. 
	By Lemmas \ref{seqtransf} and \ref{modsimple}, we may now assume that we are in the following case:
	\begin{enumerate}
		\item $H$ is a Laurent polynomial with no negative powers of any of the last $k$ variables $y_1, ..., y_k$.
		\item The sequence $\sequ$ converges to a point $p$ in the face at infinity corresponding to the face $F$ of codimension $k$ that is contained in the intersection of the last $k$ coordinate hyperplanes.
		\item For each $j \in \{1, ..., k\}$, $H$ has at least one term $c_{\mathbf{m}_j} \mathbf{z}^{\mathbf{m}_j}$ that involves the variable $y_j$ to a strictly positive power (and possibly also the first $d-k$ variables, including negative powers) but does not involve $y_{j'}$ for any $j' \neq j$.
		\item The intersection of $Q$ with the coordinate hyperplane $\{y_j = 0\}$ is a facet $F_j$, and $F$ is the intersection of the $k$ facets $F_1, ..., F_k$ and (crucially) no proper subset of them.
	\end{enumerate}
	
	As we have seen in Sections \ref{transfpoly} and \ref{secsimple}, if $p$ lies in any face satisfying the modified simple condition (in particular, if $F$ has codimension 1 or 2), then $\mathbf{z}^{-\mathfrak{v}} H$ can be monomially transformed to a polynomial $\overline{H}$ for which there is a sequence $\left\{\mathbf{w}_n\right\}_{n=1}^{\infty}$ (playing the role of $\seq$) that makes the above statements true. In this simplified case, we will find that $\mathbf{w}_n$ converges in $\mathbb{C}^d$ to a point $Z$.
	
	\begin{lem}
		\label{affineconv}
		Suppose that $H$ and $\sequ$ satisfy conditions 1-4 above. Then $\seq$ converges to a point $Z = (X_1, ..., X_{d-k}, 0, ..., 0) \in \mathbb{C}^d$ whose first $d-k$ components are nonzero and whose last $k$ are zero.
	\end{lem}
	\begin{proof}
		In this case, we have that $\sigma_F$ is all of $\mathbb{R}^{d-k} \times \mathbb{R}_{\geq 0}^k$, so by Lemma \ref{monomials}, the monomials $x_1, ..., x_{d-k}$ evaluated at $\seq$ (in other words, the first $d-k$ components of $\seq$) converge to finite nonzero values (call them $X_1, ..., X_{d-k}$), and the monomials $y_1, ..., y_k$ evaluated at $\seq$ (in other words, the last $k$ components of $\seq$) converge to zero. (If $F$ had not satisfied the modified simple condition, we could perhaps still have performed a monomial transformation, but we would not be guaranteed this convergence of each $y_j$ to zero.) Therefore, the sequence $\seq$ converges in $\mathbb{C}^d$ to a point $Z = (X_1, ..., X_{d-k}, 0, ..., 0)$.
	\end{proof}
	
	By Assumption (1.) above, we know that $H$, $\grad H$, and $\grad_{\log} H$ are all continuous at $Z$. We know that $H(Z) = 0$ because $\sequ \subseteq \sing$ converges to $Z$. No further rescaling is necessary for the log-gradient because the origin is now in $F$, so if $\grad_{\log} H(Z) \neq \mathbf{0}$, then Theorem \ref{genericresult} would already have given us that the limiting log-gradient direction at $p$ is unique and parallel to $F$. (We could show it more easily now in our special case, but Theorem \ref{genericresult} does not depend on whether or not $F$ satisfies the modified simple condition.)
	
	We now investigate the case where $\grad_{\log} H(Z) = \mathbf{0}$. In this case, it is not hard to see that $\grad H(Z)$ must be perpendicular to $F$: for $j \in \{1, ..., d-k\}$, the $j^{\mathrm{th}}$ component of $\grad_{\log} H(Z)$ is $X_j$ (a finite nonzero number) times the corresponding component of $\grad H(Z)$, so if $\grad_{\log} H(Z) = \mathbf{0}$, then the only entries of $\grad H(Z)$ that can be nonzero are the last $k$. This, of course, does not rule out the possibility that the gradient of $H$ could vanish at $Z$. However, recall that the example in Section \ref{unexpected} had a heighted CPAI in a direction not parallel to any face of $Q$, even though the gradient of $H$ at $Z = (1,0,0)$ is (0,-1,1). (Of course, $\grad_{\log} H$ at the same point is the zero vector; otherwise, Theorem \ref{genericresult} would have implied that the limiting log-gradient direction must be unique and parallel to $F$.) Therefore, we will still be answering interesting questions if we make the further assumption that the gradient does not vanish at $Z$. However, it would be interesting to see when it is possible to ``resolve" examples with singularities at the limit point into examples where the variety is smooth there (such as the paraboloid). We will assume, therefore, that $\grad_{\log} H(Z) = \mathbf{0}$ but $\grad H(Z) \neq \mathbf{0}$.
	
	\begin{thm}
		\label{nongeneric}
		Suppose that $H$ and $\sequ$ satisfy conditions 1-4, and that at the point $Z$, $\grad_{\log} H = \mathbf{0}$ but $\grad H \neq \mathbf{0}$. If the Jacobian (on $\sing$) of $\grad_{\log} H$ at $Z$ is of full rank $d-1$, then the space of limiting log-gradient directions of $H$ for sequences converging to $Z$ (that is, the set of directions for which $p$ is a CPAI) has codimension 1 and includes all directions parallel to $F$; consequently, if $F$ is not a facet, then also some the set of directions for which $p$ is a CPAI includes some directions not parallel to $F$.
	\end{thm}
	
	\begin{proof}
		By the implicit function theorem, $V(H)$ can be locally parameterized near $p$ by its tangent space at $Z$; in other words, for points $(\zz + Z) \in \sing$ sufficiently close to $Z$, the component of $\zz$ that is parallel to $\grad H$ can be written as a function $G$ of the $(d-1)$-dimensional component perpendicular to $\grad H$, and the directional derivatives of $G$ in each of these $d-1$ directions is zero at $\zz = \mathbf{0}$ (like the vertex of a paraboloid). Let $B$ be an orientation-preserving orthogonal matrix whose first $d-k$ columns are the first $d-k$ standard basis vectors (because those directions are always orthogonal to $\grad H(Z)$ when $\grad_{\log} H(Z) = \mathbf{0}$), the next $k-1$ columns are an orthonormal basis for the remaining directions perpendicular to $\grad H(Z)$, and the last column is $\grad H(Z)$ divided by its norm. 
		
		We now examine $\grad_{\log} H$ evaluated at $(B\zz + Z)$, where $$\zz = (W_1, ..., W_{d-1}, G(W_1, ..., W_d-1))$$ is such that $(B\zz + Z) \in V(H)$. Now $\grad_{\log} H(B\mathbf{z} + Z)$ can be expanded in a Taylor series as a function of $\mathbf{w} = (W_1, ..., W_{d-1})$ to first order about $\mathbf{w} = \mathbf{0}$ to give that $\grad_{\log} H(B\mathbf{z} + Z)$ is locally $\mathbf{0} + J \mathbf{w}$ (plus terms whose magnitude approaches $\mathbf{0}$ faster than $J\mathbf{w}$), provided that the $d$-by-($d-1$) Jacobian $J$ is of full rank $d-1$ (has linearly independent columns) so that $J\mathbf{w}$ is nonzero for small nonzero $\mathbf{w}$ and gives a well-defined limiting log-gradient direction when $\mathbf{w}$ approaches $\mathbf{0}$ in a given direction.  In this case, the set of limiting log-gradient directions for sequences $\zz_n$ approaching $Z$ (and therefore the set of directions for which $p$ is a CPAI) is the set of directions in the column space of $J$, which is a set of codimension 1. 
		
		Now we need to see why $p$ is a CPAI in every direction parallel to $F$. More specifically, the first $d-k$ columns of $J$ are all parallel to $F$ (and, being linearly independent, they span all the directions parallel to $F$). To see why the last $k$ components of the first $d-k$ columns of $J$ are all zero, notice that each of the last $k$ components of $\grad_{\log} H$ is of the form $y_j \frac{\partial H}{\partial y_j}$, and when evaluating at $(B\zz + Z)$, that $y_j$ factor becomes a linear combination of $W_{d-k+1}, ..., W_{d-1}, G$ with no constant term. This clearly evaluates to zero at $\mathbf{w} = \mathbf{0}$, so if its derivatives with respect to each of $W_1, ..., W_{d-k}$ do as well, then we are done by the product rule. Terms involving $G$ have first derivatives with respect to all the $W_j$ variables equal to zero, and the remaining terms all have a factor of $W_j$ for some $j \geq d-k+1$; these do not depend on $W_1, ..., W_{d-k}$ and are evaluated to zero.
		
		If $F$ is a facet ($k - 1 = 0$), then this shows that $p$ is a CPAI for all the directions parallel to $F$, and no other directions.
	\end{proof}
	
	For concreteness, I will briefly compute the space of limiting log-gradient directions in the example in Section \ref{unexpected}. We know what the answer should be: all directions in the span of [1,0,0] and [0,-1,1]. In this example, there is no need to perform a monomial transformation because $F$ (of codimension $k=2$) is already along the $x$-axis. We see that $\grad_{\log} H(1,0,0) = \mathbf{0}$ and that $\grad H(1,0,0) = (0,-1,1)$ is perpendicular to $F$ as expected. The matrix $B$ can therefore be taken to be $\begin{bmatrix} 1 & 0 & 0 \\ 0 & u & -u \\ 0 & u & u \end{bmatrix}$, where $u = \frac{\sqrt{2}}{2}$. In this example, we can parameterize $\sing$ explicitly (namely, $-uy + uz = u(x-1)^2$ for $(x,y,z) \in \sing$, so $G(W_1, W_2) = u W_1^2$ does not even depend on $W_2$), but this is not necessary; we only need the basic fact that $G$ is a function whose value and first derivatives with respect to $W_1$ and $W_2$ are zero. We have that $B\zz + Z = (W_1 + 1, uW_2 - uG, uW_2 + uG)$, so $$\grad_{\log} H(B\mathbf{z} + Z) = (-2(W_1 + 1)W_1, -uW_2 + uG, uW_2 + uG)$$ has Jacobian matrix $$J = \begin{bmatrix} -2 & 0 \\ 0 & -u \\ 0 & u \end{bmatrix},$$ so that $p$ is indeed a CPAI for precisely the codimension-1 set of directions spanned by [1,0,0] and [0,-1,1]. (The upper right entry of $J$ happened to be zero in this example, but this is not always the case; this occurred because the first component of the log-gradient happened to depend only on $x$. The factor in the first component of $\grad_{\log} H(B\mathbf{z} + Z)$ that arises from the factor of $x$ in $x \frac{\partial H}{\partial x}$ is the $(W_1 + 1)$, not the $W_1$.)
	
	As an immediate corollary, we have slightly expanded the (already generic) set of polynomials $H$ for which we can determine the directions for which a point $p \in X_A \backslash \Phi((\mathbb{C}^*)^d)$ is a critical point at infinity.
	
	\begin{thm}
		\label{simplecaseresult}
		Suppose that $\sequ \subseteq \sing^*$ has images $\Phi(\seq)$ that converge to a point $p$ in the interior of the face at infinity in $X_A$ corresponding to a face $F$ of $Q$ such that $\sigma_F/L_F$ is simplicial, and that the directions of $\grad_{\log} H(\seq)$ in $\mathbb{CP}^{d-1}$ converge to some direction $R$. For $N$ a suitable monomial transformation matrix given by Lemma \ref{modsimple}, let $\overline{H} = \tau_N^*(\zz^{-\mathfrak{v}} H)$, let $\sequw$ be a sequence given by Lemma \ref{seqtransf}, and let $Z \in \mathbb{C}^d$ be the limit of $\mathbf{w}_n$ given by Lemma \ref{affineconv}. If $\zz_n^{-\mathfrak{v}} \grad_{\log} H$ converges to the zero vector (so that Theorem \ref{genericresult} gives no conclusion), but $\grad \overline{H}(Z) \neq \mathbf{0}$ and the Jacobian $J$ of $\grad_{\log} H(Z)$ on $V(\overline{H})$ is of full rank $d-1$, then $p$ is a CPAI for a codimension-1 set of directions that contains all directions parallel to $F$.
	\end{thm}
	\begin{proof}
		Recall that, for points on $\sing^*$, we have by Lemma \ref{monomgradlog} that $\zz_n^{-\mathfrak{v}} \grad_{\log} H = \grad_{\log} (\zz_n^{-\mathfrak{v}} H)$, so by Lemma \ref{transfpoly}, $$\grad_{\log} \overline{H}(\seqw) = N^T \grad_{\log} (\zz_{m_n}^{-\mathfrak{v}} H(\zz_{m_n})) = \zz_{m_n}^{-\mathfrak{v}} N^T \grad_{\log} H(\zz_{m_n}).$$ If $\zz_n^{-\mathfrak{v}} \grad_{\log} H$ converges to the zero vector, then $\grad_{\log} \overline{H}(\seqw)$ must as well. Applying Theorem \ref{nongeneric}, the space of limiting log-gradient directions of $\overline{H}$ for sequences converging to $Z$ has codimension 1 and includes all directions parallel to $N^T F$ (that is, to the intersection of the last $k$ coordinate hyperplanes), so by conclusion 2(c) of Lemma \ref{seqtransf}, $p$ (before the monomial transformation) is a CPAI for a codimension-1 set of directions that contains all directions parallel to $F$.
	\end{proof}
	
	\begin{remark}
		Theorem \ref{simplecaseresult} works best in cases when $N$ is unimodular. When $N$ is not unimodular, the vectors formed by exponents of the $y_j$ variables of terms in $\overline{H}$ lie in a proper sublattice of $\mathbb{Z}^k$, which therefore cannot contain all $k$ standard basis vectors. If it happens not to contain \textit{any} of them, then the gradient of $\overline{H}$ at $Z$ vanishes for structural reasons.
	\end{remark}

\section{Future Directions}

There are a couple of natural potential extensions of this work that will be left for future research. First, rather than viewing critical points at infinity as obstructions to asymptotic analysis, it may be possible to analyze the asymptotic contribution of a CPAI just as we would for an affine critical point, at least in certain simple cases. One possible route to this would be to perform a monomial transformation as in the previous section, reduce the number of variables in the monomially transformed generating function to $d-k$ by setting the last $k$ variables to zero, and compute the asymptotic contribution of the affine critical points of the resulting generating function. This appears most promising when $F$ has codimension 1 and $\mathbf{r}$ is parallel to $F$, but there is still room for technical issues in degenerate cases because the asymptotics in a direction is defined to be the limiting asymptotics for exponent vectors in ALL nearby directions, not just those that are exactly parallel to $F$.

Another promising future research direction is to investigate critical points at infinity for stratified (non-smooth) manifolds that are not eventually smooth as the variety approaches a given face at infinity in the compactification. In ACSV, non-smooth varieties are handled through stratification, or partitioning into a disjoint union of finitely many manifolds (``strata") of different dimensions (and possibly some isolated points, called 0-dimensional strata). A very rough outline of computing directions in which CPAI's can occur for a sequence of points lying in a $(d-k)$-dimensional stratum $S$ of $\sing$ might look like the following:

\begin{enumerate}
	\item Write the stratum $S$ locally as an intersection of $k$ transversely intersecting algebraic hypersurfaces $V(H_j)$; if a Whitney stratification (see Appendix C of \cite{PW-book}) was chosen, then there should be such a decomposition that works for sufficiently large $n$.
	\item Any vector normal to $S$ at $\zz_n$ is a linear combination of the log-gradients of the functions $H_j$ at $\zz_n$.
	\item Use the smooth case to analyze the limiting directions for the log-gradient of each $H_j$ as $n$ grows large.
\end{enumerate}

There are at least a couple of difficulties with this approach: First and foremost, item 3 produces limiting log-gradient directions that depend on the Newton polytope of $H_j$ rather than that of $H$. A less obvious concern is that the $k$ hypersurfaces that intersect transversely at each $\zz_n$ may approach non-transversality in the limit (meaning that their log-gradients are linearly independent at each $\zz_n$ but approach being linearly dependent). In this case, we are no longer guaranteed that any projective limit of a sequence $\mathbf{r}_n$ of normals to $S$ at $\zz_n$ (where $\mathbf{r}_n$ is a linear combination of the log-gradients of the $H_j$ at $\zz_n$) will be parallel to the span of the projective limits of the $\grad_{\log} H_j(\zz_n)$; this is because the sequence of normals could converge to the zero vector and have projective directions that do not converge to be in the smaller span of the projective limits of the $\grad_{\log} H_j(\zz_n)$. Therefore, the generalization of the methods outlined for smooth $\sing^*$ to stratified $\sing^*$ will also have to be left for future research.

\bibliographystyle{alpha}
\bibliography{RPbib}

\end{document}